\documentclass[11pt,reqno]{amsart}
\usepackage[a4paper]{anysize}\marginsize{2cm}{2cm}{2cm}{2cm}
\pdfpagewidth=\paperwidth \pdfpageheight=\paperheight
\allowdisplaybreaks
\usepackage[english]{babel}
\usepackage[latin1]{inputenc}
\usepackage[T1]{fontenc}
\usepackage{amssymb,amsmath,amstext,amsfonts,amsthm}
\usepackage{mathtools}
\usepackage{verbatim}
\usepackage{relsize}
\usepackage[colorlinks=true, urlcolor=blue, linkcolor=blue, citecolor=blue]{hyperref}
\usepackage{fancyvrb}
\usepackage{tikz-cd}
\usepackage[percent]{overpic}
\usepackage{todonotes}
\usepackage[normalem]{ulem} % for striking out

\definecolor{OliveGreen}{rgb}{0,0.6,0}

\numberwithin{equation}{section}
\theoremstyle{plain}
\newtheorem*{theorem*}{Theorem}
\newtheorem{theorem}{Theorem}
\numberwithin{theorem}{section}
\newtheorem{proposition}[theorem]{Proposition}
\newtheorem{lemma}[theorem]{Lemma}
\newtheorem{corollary}[theorem]{Corollary}

\newtheorem{remark}[theorem]{Remark}

\theoremstyle{definition}
\newtheorem{definition}[theorem]{Definition}
\newtheorem{example}[theorem]{Example}
\newtheorem{notation}[theorem]{Notation}

\newcommand{\C}{\mathbb{C}}

\newcommand{\Z}{\mathbb{Z}}
\newcommand{\PP}{\mathbb{P}}

\newcommand{\R}{\mathbb{R}}

\newcommand{\mR}{\mathsmaller{\mathbb{R}}}

\newcommand{\red}[1]{{\color{red}#1}}

\makeatletter
\@namedef{subjclassname@2020}{%
	\textup{2020} Mathematics Subject Classification}
\makeatother

\date{}

\begin{document}

\author{Kaie Kubjas}
\author{Olga Kuznetsova}
\author{Luca Sodomaco}
\address{Department of Mathematics and Systems Analysis, Aalto University, Espoo, Finland}
\email{kaie.kubjas@aalto.fi}
\email{olga.kuznetsova@aalto.fi}
\email{luca.sodomaco@aalto.fi}


\title[Algebraic degree of optimization over a variety]{Algebraic degree of optimization over a variety \\ with an application to $P$-norm distance degree}

\begin{abstract}
We study an optimization problem with the feasible set being a real algebraic variety $X$ and whose parametric objective function $f_u$ is gradient-solvable with respect to the parametric data $u$.
This class of problems includes Euclidean distance optimization as well as maximum likelihood optimization.
For these particular optimization problems, a prominent role is played by the ED and ML correspondence, respectively.
To our generalized optimization problem we attach an optimization correspondence and show that it is equidimensional.
This leads to the notion of algebraic degree of optimization on $X$.
We apply these results to $p$-norm optimization, and define the $p$-norm distance degree of $X$, which coincides with the ED degree of $X$ for $p=2$. Finally, we derive a formula for the $p$-norm distance degree of $X$ as a weighted sum of the polar classes of $X$ under suitable transversality conditions.
\end{abstract}

\maketitle

%\tableofcontents

\section{Introduction}\label{sec: intro}

An optimization problem over a variety with the objective function depending on data $u \in \R^n$ takes the form
\begin{equation} \label{optimization-problem}
\begin{aligned}
\min_x \  & f_u(x) \\
\textrm{s.t.} \  & g_1(x) = \ldots = g_s(x) = 0\,,
\end{aligned}
\end{equation}
where $g_1,\ldots,g_s$ are polynomials. The feasible region of~(\ref{optimization-problem}) is the variety $X^\mR=\mathbb{V}(g_1,\ldots,g_s)$. In this paper, we consider optimization problems over varieties where the objective function $f_u$ is restricted to the class of functions that we call {\em gradient-solvable}, i.e., functions whose partial derivatives with respect to $x_i$ are rational functions and are solvable in $u$. We will make this precise in Section~\ref{sec: def optimization degree}. This class includes common objective functions such as the squared Euclidean distance and the log-likelihood function, but also the $p$-th power of the $p$-norm for an integer $p\ge 2$.

To find the optimal solution of the optimization problem~(\ref{optimization-problem}), we will consider the critical points of $f_u$ restricted to $X$. These critical points contain the global optimum of~(\ref{optimization-problem}), assuming that the global optimum exists. Since the gradient of $f_u$ is a rational function and $X^\mR$ is defined by polynomials, this problem can be studied algebraically. In order to apply classical results of algebraic geometry, we consider the complex variety $X$ whose ideal is $I(X)\coloneqq\mathcal{I}(X^\mR)\subset\C[x_1,\ldots,x_n]$, allow the function $f_u$ to take complex values and consider the complex critical points of $f_u$ on $X$. We allow only the critical points that are smooth points of $X$. We will show that the number of complex critical points that are smooth on $X$ is constant for a general data vector $u$. Following~\cite{nie2009algebraic}, we call this invariant the {\em algebraic degree} of the optimization problem~(\ref{optimization-problem}).

The algebraic degree of optimization measures the algebraic complexity of the optimal solution of the optimization problem~(\ref{optimization-problem}). In this paper, we study the properties of this degree focusing on the case when $f_u$ is the $p$-th power of the $p$-norm. In this case, we call it the {\em $p$-norm distance degree}. We provide formulas for the $p$-norm distance degree using tools from algebraic geometry. 

We assume that $\R^n$ is equipped with the standard inner product $x\cdot y=\sum_{i=1}^nx_iy_i$. Using Lagrange multipliers, the critical points of $f_u$ restricted on the smooth locus of $X$ are precisely the $x \in X_{sm}$ such that $\nabla f_u$ is perpendicular to the tangent space $T_x X$ of $X$ at $x$. If $I(X)=\langle g_1,\ldots,g_s \rangle$, then this is equivalent to $\nabla f_u$ being in the row span of the Jacobian $\text{Jac}(g_1,\ldots,g_s)$ at $x$. 

\begin{example}\label{ex: 4-norm optimization ellipse}
Let $X\subset\C^2$ be the ellipse defined by $x_1^2+4x_2^2-1$ and $f_u= (u_1-x_1)^4+(u_2-x_2)^4$. For a general data point $u=(u_1,u_2)\in\C^2$, there are 8 complex restricted critical points of $f_u$ on $X$. They form the zero locus of the ideal
$$
\langle x_1^2+4x_2^2-1, 4x_2(u_1-x_1)^3-x_1(u_2-x_2)^3\rangle \subset \C[x]\,.
$$
If we pick the point $u=(0,0)$, then all 8 critical points are real and are shown in Figure \ref{fig: crit points 4-norm ellipse}.
\begin{figure}[ht]
	\centering
	\includegraphics[width=3.3in]{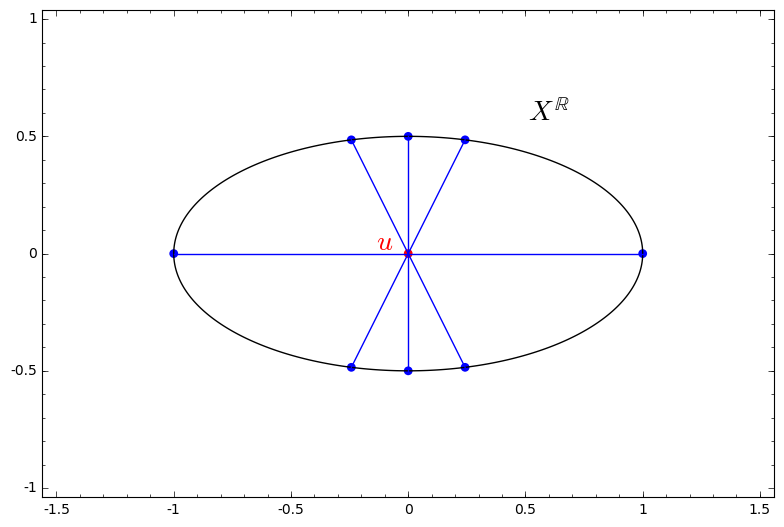}
	\caption{The 8 real critical points of $f_u$ on the ellipse $X^\mR$.}\label{fig: crit points 4-norm ellipse}
\end{figure}
\end{example}

Example \ref{ex: 4-norm optimization ellipse} is an instance of $p$-norm distance optimization from a data point $u\in\R^n$ on a real variety $X^\mR\subset\R^n$, when $p$ is a positive integer. Figure \ref{fig: many p-balls} gives an intuitive idea of how the closest point $x\in X^\mR$ from the data $u$ changes as long as $p$ increases, since the shape of the corresponding $p$-norm ball centered at $u$ and tangent to $X^\mR$ tends to a square.

\begin{figure}
    \centering
    \includegraphics[width=3.3in]{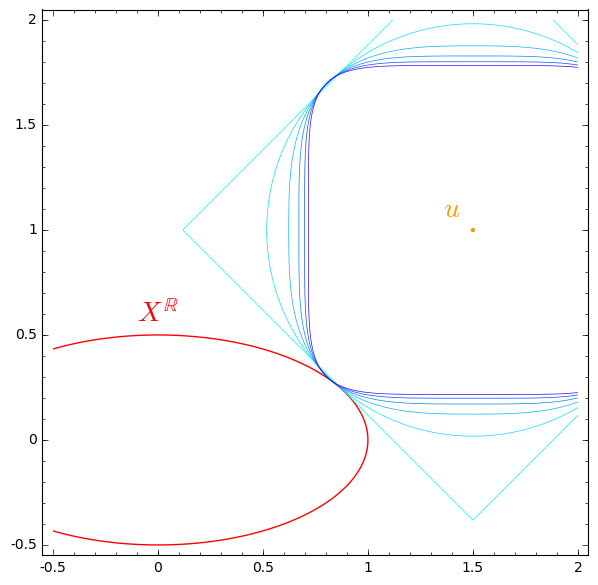}
    \includegraphics[width=3.3in]{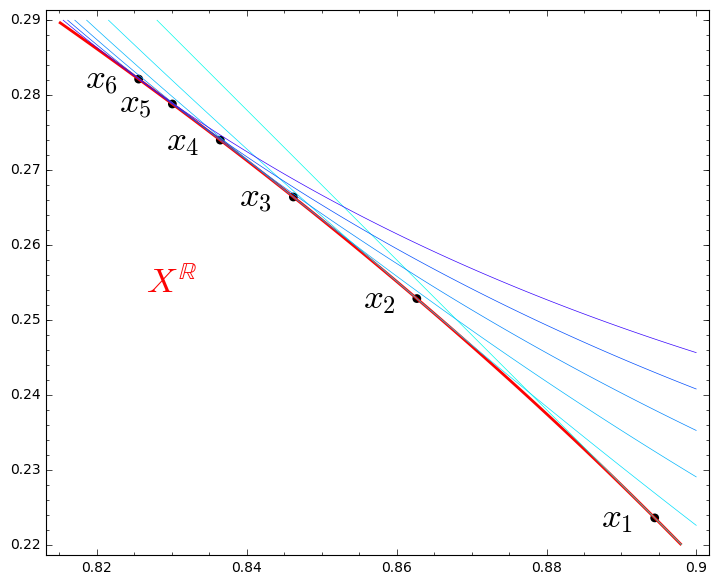}
    \caption{The $p$-norm balls centered at $u$ and tangent to the ellipse $X^\mR$. On the right, a detail of the tangency points $x_p$, which attain the minimum over $X^\mR$ of the $p$-norm from $u$, for $p\in\{1,\ldots,6\}$.}
    \label{fig: many p-balls}
\end{figure}

The two special cases that have been extensively studied are the Euclidean distance (ED) degree corresponding to the squared Euclidean distance function $\sum_i (u_i-x_i)^2$ and maximum likelihood (ML) degree corresponding to the log-likelihood function $\sum_i u_i \log x_i$.  The ED degree was studied in~\cite{DHOST} and appeared in several areas of applied algebraic geometry \cite{DLLO2017,lee2017,AHEDdegree,maximDefect,MRWmultiview,breiding2020euclidean}. The ML degree was introduced in~\cite{catanese2006maximum,hosten2005solving} and a comprehensive overview is given in~\cite{huh2014likelihood}. Our results extend the results on the ED degree and ML degree. 

The algebraic degrees of linear optimization and semidefinite optimization are investigated in~\cite{rostalski2012chapter} and~\cite{nie2010algebraicsemidefinite} respectively. Ranestad and Nie define the algebraic degree of polynomial optimization in~\cite{nie2009algebraic}. They consider optimization problems with polynomial objective functions and feasible regions given by polynomial equations and inequalities under the assumption that all the polynomials are generic. In our paper, we do not make this assumption. Moreover, although the gradient-solvable functions do not include all polynomials, they include functions that are not polynomial. In~\cite{horobet2020data}, Horobe\c{t} and Rodriguez also studied critical points of optimization problems on varieties with objective functions that are not necessarily polynomial. The class of objective functions that they consider is strictly included in the class of gradient-solvable functions.

The outline of the paper is the following: In Section~\ref{sec: prelim radical}, we study radical parametrizations restricted to varieties building on the work of Sendra, Sevilla and Villarino~\cite{sendra2017algebraic}. This allows us to rigorously define the algebraic degree over a variety for gradient-solvable functions in Section~\ref{sec: def optimization degree}. In the same section, we define the optimization correspondence $\mathcal{F}_X$ which is the variety of points $(x,u) \in \C^n \times \C^n$ such that $x\in X$ is critical for $f_u$. The optimization correspondence will be our main tool for proving formulas for the $p$-norm distance degree. In Section~\ref{sec: p-norm distance degree}, we define the $p$-norm distance degree of an affine variety $X$ and we focus on the case when $X$ is an affine cone, namely a projective variety. In Section~\ref{sec: p-norm polar classes} we introduce the $s$-conormal variety of a projective variety $X$ and use it to derive a formula for the $p$-norm distance degree of $X$ as a weighted sum of the polar classes of $X$, when $X$ is sufficiently general. This leads to several $p$-norm distance degree formulas for hypersurfaces as well as toric varieties like Segre-Veronese products of projective spaces.

The code for computations in this paper can be found at \url{https://github.com/olgakuznetsova/algebraic-degree-of-optimization}.

\section{Radical parametrizations}\label{sec: prelim radical}

\subsection{The variety of a radical parametrization} \label{sec:variety-of-a-radical-parametrization}
In this subsection, we will consider varieties that are Zariski closures of radical parametrizations, i.e., parametrizations $y=(y_1(t),\ldots,y_r(t))$ where $t=(t_1,\ldots,t_n)$ are parameters and $y_i$ are functions that contain sums, differences, products, divisions and roots of any index. The exposition in this subsection is based on \cite{sendra2017algebraic}.

\begin{definition}[Definition 2.1 in \cite{sendra2017algebraic}]
Let $t=\{t_1,\dots,t_n\}$ be a set of variables. A radical tower over $k(t)$ is a tower of fields  $\mathbb{T}=[k_0=k(t)\subset k_1 \subset \dots \subset k_m]$ such that $k_0=k(t)$, $k_{i+1}=k_i(\delta_{i+1})=k_0(\delta_1,\delta_2,\dots,\delta_{i+1})$ and $\delta_{i}^{d_{i}}=\alpha_i$ for some $\alpha_i \in k_{i-1}$ and $d_{i}\in \mathbb{N}$.
\end{definition}

Without loss of generality, we can assume that all $\delta_i \neq 0$ because otherwise $k_i=k_{i-1}$ and can be omitted.

\begin{remark}
For each $\delta_{i}^{d_{i}}=\alpha_i \in k_{i-1}$, define
\begin{equation}\label{def: Delta}
    \Delta_i^{d_{i}}=\alpha_i(t,\Delta_1,\dots,\Delta_{i-1}),
\end{equation}
where each $\Delta_i$ is a variable dependent on $t$. These new variables allow to take derivatives of the elements $\delta_i$ with respect to any $t_i$ recursively~(\cite[Chapter 8.5, Theorem 5.1]{lang2002graduate} and~\cite[Remark 2.3]{sendra2017algebraic}). Specifically, in
\[
d_i \Delta_i^{d_i-1}\frac{\partial \Delta_i}{\partial t_j }=\frac{\partial \alpha_i}{\partial t_j }\,,
\]
where the right-hand side depends only $t,\Delta_1,\dots,\Delta_{i-1}$. Backward substituting $\delta$'s makes the relation between $\frac{\partial \delta_i}{\partial t_j }$ and the previous partial derivatives explicit. 
\end{remark}

\begin{definition}[Definition 2.2 in \cite{sendra2017algebraic}]
Let $t=\{t_1,\dots,t_n\}$ be a set of variables, $\mathbb{T}$ a radical tower over $k(t)$ and $y(t)=(y_1(t),\dots,y_r(t))$ a tuple in the last field of $\mathbb{T}$. If the Jacobian of $y(t)$ has rank $n$, then $y(t)$ is called a radical parametrization. 
\end{definition}

It is not sufficient to select an $r$-tuple of elements of $k_m$ to define a radical parametrization as not all choices of branches give a well-defined map, which is illustrated in~\cite[Example 2.4]{sendra2017algebraic}. A compact way to express the choice of branch for a given $\delta_i$ is to pick a point $t_0 \in \mathbb{C}$ such that no two branches of $\delta_i$ has the same $\delta_i(t_0)$, i.e., $t_0$ is not a branch point, and fix $\delta_i(t_0)=a_i$. Hence, we can write a radical parametrization as 
\[
\mathcal{P}=\{y(t),\mathbb{T},\delta_1(t_0)=a_1,\dots,\delta_m(t_0)=a_m\}\,,
\]
where $\mathbb{T}$ is the corresponding radical tower and $t_0$ is not a branch point for any $\delta_i$.

\begin{definition}[Definitions 3.2 and 3.4 in \cite{sendra2017algebraic}] \label{defn: phi and incidence variety}
Given a radical parametrization $\mathcal{P}$, let 
\[
\phi\colon\C^n \to \C^{n+m+r}\,,\quad\phi(t)\coloneqq(t,\delta,x)
\]
and define the \emph{incidence variety} $\mathcal{B}_{\mathcal{P}}\coloneqq\overline{\operatorname{Im}{\phi}}$~\cite[Definition 3.2]{sendra2017algebraic}. Let 
\[
\pi\colon\mathcal{B}_{\mathcal{P}} \subset\C^{n+m+r} \to \C^r\,,\quad\pi(t,\delta,x)\coloneqq x
\]
be the projection on $x$.
The variety $\mathcal{V}_{\mathcal{P}}\coloneqq\overline{\operatorname{Im}{\pi}} \subset \C^r$ is called \emph{the radical variety} of $\mathcal{P}$.
\end{definition}

Going forward we will assume that $r > n$. If $r < n$ then the Jacobian condition cannot be fulfilled as its rank is at most $r$. If $r=n$, then the radical variety is the entire space.

In what follows, we will need additional notation introduced in~\cite{sendra2017algebraic}.
\begin{notation}
First, let us introduce additional variables $T$, $\Delta$, $Y$ and $Z$ and a polynomial ring $\C[T,\Delta,Y,Z]$. Each new variable other than $Z$ represents a corresponding variable in the radical tower $\mathbb{T}$ and $Z$ is used for removing the zero set of denominators. Since $\alpha(t, \delta)$ and $y$ are rational functions, we have the relations
\begin{align}{\label{new ring}}
\begin{split}
    \Delta_i^{d_i}&=\alpha_i(T,\Delta_1,\dots,\Delta_{i-1})=\frac{\alpha_{iN}(T,\Delta_1,\dots,\Delta_{i-1})}{\alpha_{iD}(T,\Delta_1,\dots,\Delta_{i-1})},\\
    Y_i&=\frac{y_{iN}(T,\Delta)}{y_{iD}(T,\Delta)}.
\end{split}
\end{align}
The rational functions in~(\ref{new ring}) are not unique, but for simplicity we denote only the dependence on $\mathcal{P}$ for the following definitions.

\noindent We will need the following polynomials
\begin{align}\label{eq:defining polynomials}
\begin{split}
    E_i&\coloneqq\Delta_i^{d_i} \cdot \alpha_{iD}(T,\Delta_1,\dots,\Delta_{i-1})-\alpha_{iN}(T,\Delta_1,\dots,\Delta_{i-1}) \text{ for } i=1,\dots,m,\\
    G_j&\coloneqq Y_j \cdot y_{jD}(T,\Delta)-y_{jN}(T,\Delta) \text{ for } j=1,\dots,r,\\
    G_Z&\coloneqq Z\cdot\operatorname{lcm}(\alpha_{1D},\cdots,\alpha_{mD},y_{1D},\cdots,y_{rD})-1.
\end{split}
\end{align}
\end{notation}

\begin{definition}
The varieties $\mathcal{D}_\mathcal{P}$ and 
$\mathcal{A}_\mathcal{P}$ are defined as 
\begin{align*}
\mathcal{D}_\mathcal{P} &\coloneqq\mathbb{V}(\mathbb{I}(E_1,\dots,E_m,G_1,\dots,G_r,G_Z))\subset \C^{n+m+r+1},\\
\mathcal{A}_\mathcal{P} &\coloneqq\mathbb{V}(\mathbb{I}(E_1,\dots,E_m,G_1,\dots,G_r,G_Z)\cap  \C[T,\Delta,Y])\subset \C^{n+m+r}\,.
\end{align*}
\end{definition}

\begin{theorem}[Theorems 3.10 and 3.11 in~\cite{sendra2017algebraic}]\label{thm: properties of Bp}
\begin{enumerate}
    \item[]
    \item There exists a unique irreducible component of $\mathcal{A}_\mathcal{P}$ that contains the incidence variety $\mathcal{B}_\mathcal{P}$.
    \item The variety $\mathcal{A}_\mathcal{P}$ is equidimensional and has dimension $n$.
    \item The variety $\mathcal{B}_\mathcal{P}$ is irreducible of dimension $n$.
    \item The radical variety $\mathcal{V}_{\mathcal{P}}$ is irreducible of dimension $n$.
\end{enumerate}
\end{theorem}

\begin{example}\label{ex: ellipse using radical parametrization}
Let $X=\mathbb{V}(g)\subset\C^2$ be the ellipse defined by $g=x_1^2+4x_2^2-1$. We consider its rational parametrization $\psi\colon\C\to\C^2$ defined by $\psi(t)=(\frac{2t}{1+t^2},\frac{1-t^2}{2(1+t^2)})$. The tangent space $T_{\psi(t)}X$ is spanned by the vector $\frac{d\psi}{dt}=-\frac{2}{(1+t^2)^2}(t^2-1,t)$. We want to optimize the objective function $f_u= (u_1-x_1)^3+(u_2-x_2)^3$ on $X$. This is achieved by imposing that $\nabla f_u$ is proportional to the normal line $N_{\psi(t)}X$. The latter is given by the left kernel of $(\frac{d\psi}{dt})^\mathsmaller{T}$, spanned by $\beta=(t, 1-t^2)$. This means that there exists $s \in \C\setminus \{0\}$ such that $(u_i-\psi_i(t))^2=s\beta_i$ for all $i\in\{1,2\}$. Solving both equations for $u_i$ by radicals yields the map $\phi\colon\C^2 \to\C^4$ defined by
\[
\phi(t,s)=(\psi_1(t),\psi_2(t),\psi_1(t)+\sqrt{st},\psi_2(t)+\sqrt{s(1-t^2)})\,.
\]
We introduce the parameters $\delta_1$ and $\delta_2$ via the relations $\delta_1^2=st$ and $\delta_2^2=s(1-t^2)$. Then we construct the radical tower
\[
\mathbb{T}\coloneqq[k_0\coloneqq\C(t,s) \subset k_1\coloneqq\C(t,s)(\delta_1)\subset k_2\coloneqq\C(t,s)(\delta_1,\delta_2)]\,.
\]
Direct computation confirms that $\phi$ satisfies the Jacobian criterion for radical parametrization. By construction of $\phi$, we want to impose that each $\delta_i$ evaluates as the principal branch. Since $(2,1)$ is not a branch point for either $\delta_i$, we can specify the evaluation of parametrization compactly as $\delta_1(2,1)=\sqrt{2}$ and $\delta_2(2,1)=\sqrt{3}\,i$. The corresponding radical parametrization is
\begin{equation*}
    \mathcal{P}=\left\{\left(\psi_1(t),\psi_2(t),\psi_1(t)+\delta_1, \psi_2(t)+\delta_2\right), \mathbb{T},\ \delta_1(2,1)=\sqrt{2},\ \delta_2(2,1)=\sqrt{3}\,i\right\}
\end{equation*}
and its radical variety $\mathcal{V}_\mathcal{P}$ is irreducible of dimension $2$.

There are $4$ possible choices of branches of $\delta_i$'s and all of them satisfy the Jacobian criterion for radical parametrization.  In this example, all four choices define the same radical variety $\mathcal{V}_{\mathcal{P}}$, although in general different choices can lead to different radical varieties. 

In Example~\ref{ex: ellipse using implicit parametrization}, we will show how to derive the same $\mathcal{V}_\mathcal{P}$ using the implicit representation of the ellipse $X$. In Section~\ref{sec: def optimization degree}, we will show that this approach provides a general algorithm for radically parametrizing the optimization correspondence $\mathcal{F}_X$.
\end{example}

\subsection{A radical parametrization restricted to a variety}\label{section: implicit varieties}

In this section, we extend the results from~\cite{sendra2017algebraic} to radical parametrizations restricted to irreducible varieties.

\begin{definition}\label{defn: restriction of radical parametrization}
Let $X \subset \C^n$ be an irreducible variety of codimension $c$   and let $\mathcal{P}$ be a radical parametrization such that there exists a dense subset of $X$ where  the rank of the Jacobian of $\mathcal{P}$ is $n$ and the Zariski closure of the zero set of $\operatorname{lcm}(\alpha_{1N},\ldots,\alpha_{mN},\alpha_{1D},\ldots,\alpha_{mD},y_{1D},\ldots,y_{rD})$ does not contain $X$. We denote by $I(X) \subset \C[T]$ the vanishing ideal of $X$ and $\phi$ is as in Definition~\ref{defn: phi and incidence variety}.
Define
\begin{align*}
    \mathcal{D}_\mathcal{P}(X)&\coloneqq (X\times \C^{m+r+1}) \cap \mathcal{D}_\mathcal{P} = \mathbb{V}(I(X)+\langle E_1,\dots,E_m,G_1,\dots,G_r,G_Z \rangle),\\
    \mathcal{A}_\mathcal{P}(X)&\coloneqq\mathbb{V}(I(\mathcal{D}_\mathcal{P}(X)) \cap \C[T,\Delta,Y]),\\
    \mathcal{B}_\mathcal{P}(X)&\coloneqq\overline{{\phi(X)}},\\
    \mathcal{V}_{\mathcal{P}}(X)&\coloneqq\overline{{\pi(\mathcal{B}_\mathcal{P}(X))}}.
\end{align*}
\end{definition}

The condition that the Zariski closure of the zero set of $\operatorname{lcm}(\alpha_{1N},\ldots,\alpha_{mN},\alpha_{1D},\ldots,\alpha_{mD},y_{1D},\ldots,y_{rD})$ does not contain $X$ implies that $\mathcal{P}$ is well-defined on a dense subset of $X$.

We will now generalize~\cite[Theorems 3.10 and 3.11]{sendra2017algebraic} to the case when the parameters are restricted to an irreducible variety $X$.

\begin{theorem}\label{thm: Bp(X) vs Ap(X)}
Let $\mathcal{P}$ be a radical parametrization. There exists a unique irreducible component of $\mathcal{A}_\mathcal{P}(X)$ that contains $\mathcal{B}_\mathcal{P}(X)$.
\end{theorem}

The main idea of the proof is to prove the statement for $\phi(\Omega)$ where $\Omega$ is a special dense subset of $X$. Since $\Omega$ is dense in $X$, the claim also holds for $\mathcal{B}_\mathcal{P}(X)$. The proof of Theorem~\ref{thm: Bp(X) vs Ap(X)} will build on Lemmas~\ref{lemma: Omega maps to nonsingular points},~\ref{lemma: X-W path connected} and~\ref{lemma: variety of D} that state some properties of the set $\Omega$. Lemmas~\ref{lemma: Omega maps to nonsingular points},~\ref{lemma: X-W path connected} and~\ref{lemma: variety of D} are analogues of Step 1, 2 and 3 of the proof of Theorem 3.10 in~\cite{sendra2017algebraic}

We start by introducing some notation as in ~\cite{sendra2017algebraic}. Let $D \subset \C^n$ be the set consisting of $t \in \C^n$ satisfying
\begin{itemize}
    \item $\alpha_{iN}(t,\delta(t))=0 \text{ or } \alpha_{iD}(t,\delta(t))=0 \text{ for some }  i \in [m],$ or
    \item $y_{jD}(t,\delta(t))=0  \text{ for some }  j \in [r].$
\end{itemize}
Define $\Omega\coloneqq X\setminus (D\cup X_{\mathrm{sing}})$. Moreover, we define the maps
\begin{align*}
    \pi_Z&\colon\mathcal{D}_\mathcal{P}(X) \to \mathcal{A}_\mathcal{P}(X)\,, \quad\pi_Z(t,\delta,y,z)\coloneqq(t,\delta,y),\\
    \varphi_Z&\colon\Omega\to\mathcal{D}_\mathcal{P}(X)\,,\quad\varphi_Z\coloneqq\pi_Z^{-1}\circ\varphi\,.
\end{align*}
The projection $\pi_Z$ is one-to-one. 

\begin{lemma}\label{lemma: Omega maps to nonsingular points}
Let $t_0 \in \Omega$. Then $p\coloneqq\varphi_Z(t_0)$ is a smooth point in $\mathcal{D}_\mathcal{P}(X)$ and is contained in a unique component of $\mathcal{D}_\mathcal{P}(X)$ of dimension $n-c$.
\end{lemma}

\begin{proof}
Let $I(X)=\langle g_1,\ldots,g_s \rangle$. Denote by $\operatorname{Jac}_p(g)$ the Jacobian of $(g_1,\ldots,g_s)$ at $p$ and by $\operatorname{Jac}_p\mathcal{D}_\mathcal{P}(X)$ the Jacobian of $\mathcal{D}_\mathcal{P}(X)$. Write the Jacobian in the block form as
\begin{equation*}
    \operatorname{Jac}_p\mathcal{D}_\mathcal{P}(X)=
   \begin{pmatrix}
    \operatorname{Jac}_p(g) & 0\\
   [\operatorname{Jac}_p\mathcal{D}_\mathcal{P}]_T & [\operatorname{Jac}_p\mathcal{D}_\mathcal{P}]_{\Delta,Y,Z}\\
   \end{pmatrix},
\end{equation*}
where $[\operatorname{Jac}_p\mathcal{D}_\mathcal{P}]_{(\cdot)}$ is the submatrix of $\operatorname{Jac}_p\mathcal{D}_\mathcal{P}(X)$ given by the partial derivatives with respect to the given variables. Since the dimension of $X$ is $n-c$ and $t_0 \in X \setminus X_{\text{sing}}$, then the rank of the upper part of the Jacobian is $c$ at $p=\varphi_Z(t_0)$. In~\cite[Proof of Theorem 3.10: Step 1]{sendra2017algebraic} it is shown that $[\operatorname{Jac}_p\mathcal{D}_\mathcal{P}]_{\Delta,Y,Z}$ is a lower diagonal matrix with the diagonal entries given by $\frac{\partial E_i}{\partial \Delta_i} = e_i(\Delta_i)^{e_i-1} \cdot \alpha_{iD}(T,\Delta_i,\ldots,\Delta_{i-1})$, $\frac{\partial G_j}{\partial Y_j} = y_{jD}(T,\Delta)$ and $\frac{\partial G_Z}{\partial Z} = \operatorname{lcm}(y_{1D},\ldots,y_{rD},\alpha_{1D},\ldots,\alpha_{mD})$. By the definition of the set $D$, the rank of  $[\operatorname{Jac}_p\mathcal{D}_\mathcal{P}]_{\Delta,Y,Z}$ at a point $p=\varphi_Z(t_0)$ where $t_0 \in X \setminus D$ is $m+r+1$. Hence, the rank of $\operatorname{Jac}_p\mathcal{D}_\mathcal{P}(X)$ at $p=\varphi_Z(t_0)$ for $t_0 \in \Omega$ is $c+m+r+1$. Since the maximal rank of $\operatorname{Jac}_p\mathcal{D}_\mathcal{P}(X)$ at a point $p=\varphi_Z(t)$ for $t \in X$ is $c+m+r+1$, by~\cite[Lecture 14, p. 174]{harris1992algebraic} $p$ is nonsingular and so by~\cite[Chapter 9.6, Theorem 8]{cox2013ideals} lies in a unique component of $\mathcal{D}_\mathcal{P}(X)$ of dimension $n-c$.
\end{proof}

\begin{lemma}\label{lemma: X-W path connected}
Let $X \subset \C^n$ be an irreducible variety and $W\subset X$ be a variety of dimension strictly less than $\dim(X)$. Then $X\setminus W$ is path-connected. 
\end{lemma}

\begin{proof}
We denote by $H\cong\PP^{n-1}$ the hyperplane at infinity of $\C^n$, so that $\PP^n=\C^n\cup H$. Let $\overline{X}\subset\PP^n$ be the projective closure of $X$. In particular $X=\overline{X}\cap\C^n$. Hence, $W\cup H$ is a variety and $X\setminus W= \overline{X} \setminus (W \cup H)$. Therefore, $X\setminus W$ is connected in the classical topology~\cite[Corollary 4.16 p.68]{mumford1995algebraic}. By the Lojasiewicz's theorem~\cite[Theorem 4]{lojasiewicz1964triangulation}, $\overline{X}$ is a finite CW complex. CW complexes are locally contractible~\cite[Proposition A.4]{hatcher2002algebraic}, so in particular, they are locally path-connected. Hence, $\overline{X} \setminus (W \cup H)=X\setminus W$ is a connected open set in a locally path-connected space, so it is path-connected.
\end{proof}

\begin{lemma}\label{lemma: variety of D}
There exists a variety $W\subset X$ of dimension strictly less than $\dim(X)$ such that the set $(D \cup X_{\mathrm{sing}}) \cap X$ is contained in $W$.
\end{lemma}

\begin{proof}
By Step 3 of the proof of~\cite[Theorem 3.10]{sendra2017algebraic}, the set $D$ is contained in an algebraic variety $W_D$ of dimension less than $n$. We choose $W_D$ to be the Zariski closure of $D$. By assumption in Definition~\ref{defn: restriction of radical parametrization}, $X$ is not contained in $W_D$. The singular locus $X_{\mathrm{sing}}$ of $X$ also has dimension strictly less than $n-c$. The irreducibility of $X$ then implies that $(D \cup X_{\mathrm{sing}}) \cap X$ has dimension strictly less than $n-c$ by~\cite[Thereom 1.19]{Shafarevich}.
\end{proof}

\begin{proof}[Proof of Theorem~\ref{thm: Bp(X) vs Ap(X)}]
For contradiction, assume that there exist two different irreducible components $\Gamma_1$ and $\Gamma_2$ of $\mathcal{A}_\mathcal{P}(X)$ that contain $\mathcal{B}_\mathcal{P}(X)$ and let $\Sigma_i\coloneqq\pi_Z^{-1}(\Gamma_i)$. By Lemma~\ref{lemma: variety of D}, there exists a variety $W \subset X$ that contains $(D \cup X_{\mathrm{sing}}) \cap X$ and whose dimension is at most $n-c-1$. Let $\Omega'\coloneqq\mathcal{X}\setminus W \subset \Omega$.  

Since $\Omega'$ is dense in $X$, there exist $t_1,t_2 \in \Omega'$ such that $p\coloneqq\varphi_Z(t_1) \in \Sigma_1$ and $q\coloneqq\varphi_Z(t_2) \in \Sigma_2$. By Lemma~\ref{lemma: Omega maps to nonsingular points}, either point lies in a unique component of $\mathcal{D}_\mathcal{P}(X)$. By Lemma~\ref{lemma: X-W path connected}, there is a path $\gamma$ in $\Omega'$ between $t_1$ and $t_2$. Then $\varphi_Z(\gamma)$ is connected by the continuity of $\varphi_Z$.

Let $U$ be the union of components of $\mathcal{D}_\mathcal{P}(X)$ that are different from $\Sigma_1$. Both $\Sigma_1 \cap \varphi_Z(\gamma)$ and $U \cap \varphi_Z(\gamma)$ are closed in $\varphi_Z(\gamma)$, because $\Sigma_1$ and $U$ are closed in their ambient topology. Furthermore, these two sets cover $\varphi_Z(\gamma)$ and are nonempty. Therefore, since $\varphi_Z(\gamma)$ is connected, they have a nonempty intersection. However, any point that is contained in the intersection of two irreducible components of a variety is singular~\cite[Chapter 9.6 Theorem 8]{cox2013ideals}, so $\varphi_Z(\gamma)$ contains singular points. This is a contradiction to Lemma~\ref{lemma: Omega maps to nonsingular points}.
\end{proof}

\begin{theorem}\label{thm: components of Ap(X)}
\begin{enumerate}
    \item[]
    \item $\mathcal{A}_\mathcal{P}(X)$ is an equidimensional variety of dimension $n-c$. 
    \item $\mathcal{B}_\mathcal{P}(X)$ is an irreducible variety of dimension $n-c$. 
    \item $\mathcal{V}_\mathcal{P}(X)$ is an irreducible variety of dimension $n-c$.
\end{enumerate}
\end{theorem} 

\begin{proof}
Consider the projection $\pi$ from $\mathcal{D}_\mathcal{P}(X)$ to $X$ given by $\pi: (t,\delta,y,z)\mapsto t$. This map has finite fibers: the polynomial $G_Z$ forces all polynomials $E_i$ and $G_j$ to be non-zero. Furthermore, each $E_i$ contains a single variable $\Delta_i$, so $E_i$ has finitely many solutions in $\Delta_i$. Similarly, each $G_j$ and $G_Z$ have finitely many solutions in  $Y_j$ and $Z$, respectively. Therefore, for any $t \in \pi (\mathcal{D}_\mathcal{P})(X)$, we have $\dim (\pi^{-1}(t))=0$.

Take any irreducible component $D_i(X)$ of $\mathcal{D}_\mathcal{P}(X)$. Since $\pi(t,\delta,y,z) \in X$ for all $(t,\delta,y,z) \in D_i(X)$, it follows that $\dim (\pi (D_i(\mathcal{X})))\leq n-c$. It is a quasi-projective variety and the map $\pi$ is regular.
To apply~\cite[Corollary 11.13]{harris1992algebraic}, take an open set $U \subset \pi(D_i(X)) \subset \overline{\pi(D_i(X))}$ and then
\[
\dim D_i(X)=\dim \pi^{-1}(U)=\dim (U)+\min\{\dim (\pi^{-1}(t))|t \in U\}\leq n-c\,.
\]

On the other hand, $D_i(X)$ is an irreducible component of the intersection of the variety $\mathcal{D}_\mathcal{P}$ of dimension $n$ and the variety $X\times \C^{m+r+1}$ of dimension $n+m+r+1-c$, so by the Affine Dimension Theorem~\cite[Proposition I.7.1]{hartshorneAlgebraic}
\[
\dim D_i(X)\geq n+(n+m+r+1-c)-(n+m+r+1) = n-c\,.
\]
Hence, every irreducible component of $\mathcal{D}_\mathcal{P}(X)$ has dimension $n-c$ and so does $\mathcal{D}_\mathcal{P}(X)$.

(1) Since $\pi_Z$ is a finite map, the dimension of every irreducible component of $\mathcal{D}_\mathcal{P}(X)$ is $n-c$ and $\mathcal{A}_\mathcal{P}(X)=\overline{\pi_Z(\mathcal{D}_\mathcal{P}(X))}$, then again by~\cite[Corrollary 11.13]{harris1992algebraic} every irreducible component of $\mathcal{A}_\mathcal{P}(X)$ has dimension $n-c$.

(2) By Theorem~\ref{thm: Bp(X) vs Ap(X)}, the variety $\mathcal{B}_\mathcal{P}(X)$ lies in a unique irreducible component of $\mathcal{A}_\mathcal{P}(X)$, which by part (1) has dimension $n-c$. Therefore, $\dim \mathcal{B}_\mathcal{P}(X) \leq n-c$. Consider the coordinate projection $\pi^*$ from $\C^{n+m+r}$ to $\C^{n}$ given by $(t, \delta, y)\mapsto t$. 
Then $\Omega \subset \pi^*(\mathcal{B}_\mathcal{P}(X)) \subset X \subset \C^n$. Since $\Omega$ is dense in $X$ by Lemma~\ref{lemma: variety of D}, then also $\pi^*(\mathcal{B}_\mathcal{P}(X))$ is dense in $X$.  
By~\cite[Proposition 5 in Chapter 9.5]{cox2013ideals}, there exists a subspace $H \subset \C^n$ of dimension $n-c$ such that the projection of $X$ onto $H$ is Zariski dense. Thus also the projection of $\pi^*(\mathcal{B}_\mathcal{P}(X))$ onto $H$ is Zariski dense and $\dim \mathcal{B}_\mathcal{P}(X) \geq n-c$. Therefore $\dim \mathcal{B}_\mathcal{P}(X)=n-c$ and $\mathcal{B}_\mathcal{P}(X)$ is an irreducible component of $\mathcal{A}_\mathcal{P}(X)$.

(3) Let $W=\overline{\operatorname{Im}(\mathcal{P}(X))}$. By assumption, there exists a dense subset of $X$, where $\mathcal{P}$ is well-defined and the rank of the Jacobian of $\mathcal{P}$ restricted to $X$  is $n$. Then by the inverse function theorem for manifolds~\cite[Theorem 5.11]{lee2013smooth}, at every smooth point $t_0$ in this dense subset of $X$, there exist connected neighbourhoods $U_0$ of $t_0$ and $V_0$ of $\mathcal{P}(t_0)$ such that $U_0$ and $V_0$ are diffeomorphic, so a dense subset of $X$ is diffeomorphic to a dense subset of  $W$ and hence $\dim W=n-c$. 
 
Since $W \subset \mathcal{V}_\mathcal{P}(X)$, $\dim \mathcal{V}_\mathcal{P}(X) \geq n-c$. On the other hand, $\mathcal{V}_{\mathcal{P}}(X)=\overline{{\pi(\mathcal{B}_\mathcal{P}(X))}}$, where $\pi$ is given in Definition~\ref{defn: phi and incidence variety} and $\mathcal{B}_\mathcal{P}(X)$ is irreducible of dimension $n-c$. Thus $\mathcal{V}_{\mathcal{P}}(X)$ is irreducible and by~\cite[Corrollary 11.13]{harris1992algebraic}, $\dim \mathcal{V}_{\mathcal{P}}(X)\leq n-c$. This proves that $\dim \mathcal{V}_{\mathcal{P}}(X)= n-c$.
\end{proof}

\begin{example}\label{ex: ellipse using implicit parametrization}
Continuing Example~\ref{ex: ellipse using radical parametrization}, let us give a corresponding radical parametrization when $X=\mathbb{V}(g)\subset\C^2$ is given implicitly by $g=x_1^2+4x_2^2-1$. As before, we want to impose the condition that $\nabla f_u$ is perpendicular to the tangent space $T_xX$. This condition is equivalent to saying that the augmented Jacobian matrix 
\begin{equation}\label{eq: rank condition}
\begin{pmatrix}
(u_1-x_1)^2 & (u_2-x_2)^2\\
x_1 & 4x_2
\end{pmatrix}
\end{equation}
has rank at most $1$ and motivates the construction of the map $\phi\colon\C^3\to\C^4$ defined by
\[
\phi(x_1,x_2,s)=(x_1,x_2,x_1+\sqrt{sx_1},x_2+\sqrt{4sx_2})\,.
\]
Here we choose the principal branch in both coordinates but, in fact, for any combination of branches, the rank of the Jacobian of the corresponding map is $3$ on a dense subset of $\C^3$. Observe that each branch corresponds to one of the four half-lines displayed on the right of Figure \ref{fig: ellipse3norm}. 

We introduce the parameters $\delta_1$ and $\delta_2$ via the relations $\delta_1^2=sx_1$ and $\delta_2^2=4sx_2$. The point $(1,1,1)$ is not a branch point for either $\delta_i$ and we can fix any choice of evaluation of the form $\delta_1(1,1,1)=\pm 1\eqqcolon a_1$ and $\delta_2(1,1,1)=\pm 2\eqqcolon a_2$.
The corresponding radical parametrization becomes
\[
\mathcal{P}=\{(x_1,x_2,x_1+\delta_1,x_2+\delta_2),\ \mathbb{T},\ \delta_1(1,1,1)=a_1,\ \delta_2(1,1,1)=a_2\}\,,
\]
where $\mathbb{T}\coloneqq[k_0 \subset k_0(\delta_1) \subset k_0(\delta_1,\delta_2)]$ is radical tower over $k_0=\C(x_1,x_2,s)$.

The pair $(\mathcal{P},X)$ satisfies the conditions from Definition~\ref{defn: restriction of radical parametrization}, i.e., the Jacobian of  $\mathcal{P}$ has rank $3$ on a dense subset of $X$ and $\mathbb{V}(sx_1x_2)$ does not contain $X$. Hence, we can restrict $\mathcal{P}$ to $X$ and the resulting variety $\mathcal{V}_\mathcal{P}(X)$ is irreducible of dimension $2$, and equals the radical variety $\mathcal{V}_\mathcal{P}$ from Example~\ref{ex: ellipse using radical parametrization}. As in Example~\ref{ex: ellipse using radical parametrization}, the radical variety $\mathcal{V}_\mathcal{P}(X)$ does not depend on the choice of evaluation of $\delta_i$'s.
\end{example}

\section{The algebraic degree of optimization over a variety}\label{sec: def optimization degree}

Let $X^\mR\subset\R^n$ be a real algebraic variety defined by the ideal $\mathcal{I}(X^\mR)=\langle g_1,\ldots,g_s\rangle\subset\R[x]$. We consider the complex variety $X$ whose ideal is $I(X)\coloneqq\mathcal{I}(X^\mR)\subset\C[x]$. Without loss of generality, we assume that $I(X)$ is a prime ideal, namely that $X$ is an irreducible variety. Moreover, let $c=\mathrm{codim}(X)$.

Let $f_u\colon\C^n\times\C^n\to\C$ be a parametric objective function determined by a data point $u\in\C^n$. 
We require the function $f_u(x)$ to be holomorphic. We study the optimization problem
\begin{equation}\label{eq: optimization problem, general form}
\min_{x\in X^\mR} f_u(x)\,.
\end{equation}

Our prototype examples are
\begin{enumerate}
    \item ED optimization, where we fix a positive definite quadratic form $q^\mR$ on $\R^n$, we extend it to a polynomial function $q\colon\C^n\to\C$ and we study the squared distance function $f_u(x)=d_u(x)=q(u-x)$ restricted to the complex zero locus $X$ of a real variety $X^\mR\subset\R^n$.
    \item ML optimization, where the objective function is the log-likelihood function $f_u(x)=\ell(x\mid u)=\sum_{i=1}^n u_i\log(p_i(x))$ and the variety $X$ is the Zariski closure of a statistical model, i.e., a real variety $X^\mR$ intersected with the probability simplex $\Delta_{n-1}$.
\end{enumerate}

Both these prototype examples satisfy the following properties.

\begin{definition}\label{def: solvable}
A holomorphic function $f_u\colon\C^n\times\C^n\to\C$ is {\em gradient-solvable in $u$} if the following conditions hold:
\begin{itemize}
    \item for all $i\in[n]$, the partial derivative $\frac{\partial f_u}{\partial x_i}$ with respect to $x_i$ is in $\C(u_i,x)$ and it depends on $u_i$, and
    \item for all $i\in[n]$ and for all $g\in\C[x]$, the equation $\frac{\partial f_u}{\partial x_i}=g$ is solvable in $u_i$ by radicals after derationalization 
\end{itemize}
\end{definition}

We denote by $\nabla f_u$ the gradient with respect to $x$. In~\cite{horobet2020data}, the authors consider objective functions such that $u$ can be expressed as a rational function in $(x,\nabla f_u)$. This class of functions is contained in the set of gradient-solvable functions.

When $f_u\colon\C^n \times \C^n \to \C$ is gradient-solvable in $u$, the problem of determining the critical points of $f_u$ for a fixed $u \in \C^n$ on an algebraic variety $X\subset\C^n$ can be tackled with algebraic tools.
A point $x\in X_\mathrm{sm}$ is {\em critical} for $f_u$ if the vector $\nabla f_u(x)$ is orthogonal to the tangent space $T_xX$. Our goal is to describe the critical points of the function $f_u$ for a general data point $u$. If $X$ is the complex zero locus of a real algebraic variety $X^\mR\subset \R^n$ and $f_u$ is the extension to the complex  numbers of a real objective function $f_u\colon \R^n\to\R$, then the set of real critical points of $f_u$ on $X$ contains the constrained local extrema of the real objective function $f_u$. 

Recall that the ideal of $X$ is $\langle g_1,\ldots,g_s\rangle\subset\C[x]$. We indicate with $g$ the vector $(g_1,\ldots,g_s)$. If $f_u$ is a gradient-solvable function, then the entries of $\nabla f_u$ are rational functions and hence the entries of the augmented Jacobian matrix
$$
J(f_u,g)=
\begin{pmatrix}
\nabla f_u\\
\mathrm{Jac}(g)
\end{pmatrix}
$$
are not necessarily polynomial. Let $\left(\frac{\partial f_u}{\partial x_i}\right)_{D}$ denote the denominator of $\frac{\partial f_u}{\partial x_i}$. We define the derationalized augmented Jacobian matrix as
$$
\tilde{J}(f_u,g)=
J(f_u,g)
\cdot
\text{diag}\left(\left(\frac{\partial f_u}{\partial x_1}\right)_{D}, \ldots, \left(\frac{\partial f_u}{\partial x_n}\right)_{D} \right)\,.
$$
The critical ideal of $X$ with respect to $f_u$ is the following ideal in $\C[x]$: 
\begin{equation}\label{eq: critical ideal sec 3}
\left(I(X)+\left\langle\mbox{$(c+1)\times(c+1)$ minors of $\tilde{J}(f_u,g)$}\right\rangle\right)\colon \left({I(X_{\mathrm{sing}})} \cdot \left\langle \prod_{i=1}^n\left(\frac{\partial f_u}{\partial x_i}\right)_{D} \right\rangle \right)^\infty\,.
\end{equation}

\begin{definition}\label{def: optimization correspondence}
Let $(f_u,X)$ be compatible as in Definition~\ref{compatible}.
The {\em optimization correspondence} of $X$ associated with $f_u$ is the variety
\[
\mathcal{F}_X\coloneqq\overline{\{(x,u)\in (X\times\C^n)\setminus H \mid\mbox{$x\in X_\mathrm{sm}$, $x$ critical point of $f_u$}\}}\,,
\]
where $H$ is defined by $\left(\frac{\partial f_u}{\partial x_1}\right)_{D}\cdots\left(\frac{\partial f_u}{\partial x_n}\right)_{D}=0$.
\end{definition}

\begin{lemma}\label{lem: ideal F-correspondence}
Consider $f_u$ and $X$ as in Definition \ref{def: optimization correspondence}. Then $\mathcal{F}_X$ coincides with the variety of the ideal \eqref{eq: critical ideal sec 3} in the larger ring $\C[x,u]$.
\end{lemma}
\begin{proof}
Consider a point $(x,u)$ such that $x\in X_\mathrm{sm}$, $(x,u)\notin H$ and $x$ is a critical point of $f_u$. Then clearly $(x,u)$ lies in the variety of \eqref{eq: critical ideal sec 3}. Taking closures we get that $\mathcal{F}_X$ is contained in the variety of \eqref{eq: critical ideal sec 3}. Conversely, let $(x,u)$ be in the variety of \eqref{eq: critical ideal sec 3}. The points which are in the complement of $(X_\mathrm{sing}\times\C^n)\cup H$ are dense in that variety, so we may assume that $(x,u)\in (X_\mathrm{sm}\times\C^n)\setminus H$. The condition in \eqref{eq: critical ideal sec 3} implies that $\nabla f_u(x)$ is in the row span of $\mathrm{Jac}(g)$, hence $x$ is a critical point of $f_u$, namely $(x,u)\in\mathcal{F}_X$.\qedhere
\end{proof}

We will apply the techniques introduced in Section \ref{sec: prelim radical} and show that, if $X$ is radically parametrized and $f_u$ is gradient-solvable in $u$, then the number of complex critical points of the function $f_u\colon X\to\C$ is described using $\mathcal{F}_X$ and is constant for general data point $u$. This will lead to the notion of algebraic degree of $X$ with respect to the function $f_u$.

Fix $I(X)=\langle g_1,\ldots,g_s\rangle\subset\C[x]$ and let $g$ be the vector $(g_1,\ldots,g_s)$. Let $(x,u)$ be a point in $\mathcal{F}_X$. Eq.~(\ref{eq: critical ideal sec 3}) implies that $(\nabla_xF)_N$ is a linear combination of the rows of the matrix
\[
\widetilde{\mathrm{Jac}}(g)\coloneqq\mathrm{Jac}(g)\cdot\text{diag}\left((\partial f_u/\partial x_1)_{D},\ldots, (\partial f_u/\partial x_n)_{D} \right)\,,
\]
and the rank of $\widetilde{\mathrm{Jac}}(g)$ is $c$. Similar to \cite[Definition 2]{grossMaximum}, we can construct a new $c\times n$ matrix $M=(m_{ij})$ of entries in $\C[x]$, where each row of $M$ is a random linear combination of the rows of $\widetilde{\mathrm{Jac}}(g)$. All in all, there exist complex numbers $s_1,\dots,s_c$ such that
\begin{equation}\label{eq: gradient vs Jacobian}
\left(\frac{\partial f_u}{\partial x_j}(x,u)\right)_N=\sum_{i=1}^c s_i\,m_{ij}(x)\quad\forall\,j\in\{1,\ldots,n\}\,.
\end{equation}
By assumption, Eq.~\eqref{eq: gradient vs Jacobian} is solvable in $u_j$ for all $j$, so for a given branch we can write
$u_j=h_{j}(x,s)$ for some $h_{j}(x,s) \in \overline{\C(x,s)}\setminus \{0\}$. We will denote the vector $[h_{1}(x,s),\dots,h_{n}(x,s)]$ by $h(x,s)$.

Define the radical tower
\[
\mathbb{T}\coloneqq[k_0\coloneqq\C(x,s) \subset \dots \subset k_m\coloneqq \C(x,s)(\delta_1,\dots,\delta_m)]\,,
\]
where $k_m$ is the smallest field that contains all $h_j(x,s)$ and $\delta_i^{d_i}=\alpha_i$ for some $\alpha_i \in k_{i-1}$ and $d_i\in \mathbb{N}$. Since $\alpha_i$ does not need to be a polynomial, we can write 
\begin{equation}\label{eq: delta generators}
    \delta_i^{d_i}=\frac{\alpha_{iN}}{\alpha_{iD}}.
\end{equation}

Let $\phi\colon \C^{n+c}\to \C^{2n
}$ be the map $(x,s)\mapsto (x_1,\dots,x_n,h_1(x,s),\dots,h_n(x,s))$. For $\phi$ to define a radical parametrization, we need to show that the Jacobian of $\phi$ has rank $n+c$ . The Jacobian of $\phi$ has block form, where the top left block is the identity matrix. It is enough to show that the lower right block, which is the Jacobian of $\phi$ with respect to $s$, has rank $c$. Fix $x \in X_{\text{sm}}$. Then from Eq.~(\ref{eq: gradient vs Jacobian}) it follows that 
\begin{equation} \label{eq: composition}
\sum_{i=1}^c s_i\,m_{ij}(x) = \nabla f_u(x) \circ \phi.
\end{equation}
The Jacobian of the left-hand side of~(\ref{eq: composition}) with respect to $s$ is $M$, which has rank $c$. Since the Jacobian of the composition of functions is the product of Jacobians, this implies that also the Jacobian of $\phi$ with respect to $s$ has rank $c$.

\begin{definition}\label{compatible}
Consider a holomorphic function $f_u\colon\C^n\times\C^n \to \C$ gradient-solvable in $u$ and let $X\subset\C^n$ be an algebraic variety. The pair $(f_u,X)$ is called {\em compatible} if for every branch of \eqref{eq: gradient vs Jacobian}, there exists a radical parametrization $\phi$ as above such that $\operatorname{lcm}(\alpha_{1N},\ldots,\alpha_{mN},\alpha_{1D},\ldots,\alpha_{mD},\phi_{1D},\ldots,\phi_{2nD})$ does not contain $X \times \C^c$. 
\end{definition}

\begin{theorem}\label{prop: X unirational, f-correspondence parametrizable}
Let $X\subset\C^n$ be an irreducible variety of codimension $c$, and consider a holomorphic function $f_u$ gradient-solvable in $u$ such that $(f_u,X)$ is compatible. The optimization correspondence $\mathcal{F}_X$ is an equidimensional variety of dimension $n$.
\end{theorem}

\begin{proof}
For a branch of \eqref{eq: gradient vs Jacobian}, consider a radical parametrization $\phi$ satisfying the conditions of Definition~\ref{compatible}, i.e., the Zariski closure of the zero set of $\operatorname{lcm}(\alpha_{1N},\ldots,\alpha_{mN},\alpha_{1D},\ldots,\alpha_{mD},\phi_{1D},\ldots,\phi_{2nD})$ does not contain $X \times \C^c$.
Choose a value $(x_0,s_0)\in X \times \C^c$ such that $h(x_0,s_0)$ is well-defined and $(x_0,s_0)$ is not a branch point for any $\delta_i$, i.e., each branch of $\delta_i$ has a different value at $(x_0,s_0)$, and specify the choice of $\delta$ that corresponds to $(x_0,s_0)$. Consider the radical parametrization
\[
\mathcal{P}\coloneqq\left\{\phi,\ \mathbb{T},\ \delta_1(x_0,s_0)=a_1, \dots, \ \delta_m(x_0,s_0)=a_m\right\}\,.
\]

The radical variety $\mathcal{V}_\mathcal{P}(X)$ of $\mathcal{P}$ is irreducible of dimension $n+c-c=n$  by Theorem~\ref{thm: components of Ap(X)}. Since it corresponds to a branch of~(\ref{eq: gradient vs Jacobian}), it is contained in $\mathcal{F}_X$. Conversely, let $(x,u)$ be a point in $\mathcal{F}_X$ such that $x\in X_\mathrm{sm}$ and $u$ is a root of~(\ref{eq: gradient vs Jacobian}). Then there exists a radical parametrization  $\mathcal{P}$ corresponding to a branch of~(\ref{eq: gradient vs Jacobian})  such that $(x,u) \in \mathcal{V}_\mathcal{P}(X)$. Hence the points $(x,u)$ in $\mathcal{F}_X$ such $x\in X_\mathrm{sm}$ and $u$ is a root of~(\ref{eq: gradient vs Jacobian}) are contained in the union of radical varieties over all branches of~(\ref{eq: gradient vs Jacobian}). Since such points are dense in $\mathcal{F}_X$, then also $\mathcal{F}_X$ is contained in the union. Therefore, $\mathcal{F}_X$ is equal to the union of radical varieties over all branches of~(\ref{eq: gradient vs Jacobian}), and it is equidimensional of dimension $n$. 
\end{proof}

It would be interesting to characterize when the variety $\mathcal{F}_X$ is irreducible. The variety $\mathcal{F}_X$ in Example \ref{ex: F_X decomposable threefold} is not irreducible, whereas in Examples \ref{ex: F_X irred ellipse} and \ref{ex: F_X cardioid different f_u} it is.

\begin{example}\label{ex: F_X decomposable threefold}
Let $n=4$ and consider the smooth cubic threefold $X\colon x_1^3+x_2^3+x_3^2x_4-1=0$ in $\C^4$. Let $f_u=\sum_{i=1}^4(u_i-x_i)^3$, namely the cube of $\|u-x\|_3$.
Using the function \verb+bertiniPosDimSolve+ of the \verb+Macaulay2+ package \verb+Bertini+, we computed a numerical irreducible decomposition of $\mathcal{F}_X\subset\C_x^4\times\C_u^4\cong\C^8$ with four components of dimension $4$ and degree $21$.
\end{example}

\begin{example}\label{ex: F_X irred ellipse}
We can see now that the radical parametrization $\mathcal{V}_\mathcal{P}(X)$ given in Example~\ref{ex: ellipse using implicit parametrization} corresponds to the radical parametrization of an irreducible component of $\mathcal{F}_X$ when $X=\mathbb{V}(x_1^2+4x_2^2-1)$ and 
$f_u= (u_1-x_1)^3+(u_2-x_2)^3$. Direct computation in \verb+Bertini+ shows that in this case $\mathcal{F}_X=\mathcal{V}_\mathcal{P}(X)$.
\end{example}

\begin{example}\label{ex: F_X cardioid different f_u}
Let $n=2$, $c=1$ and $X=\mathbb{V}(g)\subset \C^2$ be the cardioid given by $g=(x_1^2+x_2^2+x_1)^2-(x_1^2+x_2^2)$.
We optimize the objective function $f_u=f_{u_1}+f_{u_2}$ on $X$, where $f_{u_i}=-1/3a_i(u_i-x_i)^3-1/2b_i(u_i-x_i)^2-c_i(u_i-x_i)-d_i$. We want to impose the condition $\mathrm{rank}(J(f_u,g))\le 1$, where up to a common factor $2$
\[
J(f_u,g)=\tilde{J}(f_u,g)=
\begin{pmatrix}
a_1(u_1-x_1)^2+b_1(u_1-x_1)+c_1 & a_2(u_2-x_2)^2+b_2(u_2-x_2)+c_2 \\
2x_1^3+2x_1x_2^2+3x_1^2+x_2^2 & x_2(2x_1^2+2x_2^2+2x_1-1)
\end{pmatrix}\,.
\]
In this case, solving by radicals the system \eqref{eq: gradient vs Jacobian} yields the solutions
\begin{align*}
    u_1&=\frac{2 a_1 x_1-b_1\pm\sqrt{b_1^2-4 a_1 c_1+4 a_1 s (2x_1^3+2x_1x_2^2+3x_1^2+x_2^2)}}{2 a_1}\\
    u_2&=\frac{2 a_2 x_2-b_2\pm\sqrt{b_2^2-4 a_2 c_2+4 a_2 s x_2(2x_1^2+2x_2^2+2x_1-1)}}{2 a_2}.
\end{align*}
Define $\delta_1$ and $\delta_2$ via the relations $\delta_1^2=b_1^2-4 a_1 c_1+4 a_1 s (2x_1^3+2x_1x_2^2+3x_1^2+x_2^2)$ and $\delta_2^2=b_2^2-4 a_2 c_2+4 a_2 s x_2(2x_1^2+2x_2^2+2x_1-1)$. We construct the radical tower
\[
\mathbb{T}\coloneqq[k_0\coloneqq\C(x,s) \subset k_1\coloneqq\C(x,s)(\delta_1)\subset k_2\coloneqq\C(x,s)(\delta_1,\delta_2)]
\]
and the parametrization map
\[
\phi\colon \C^3 \rightarrow \C^4\,,\quad\phi(x_1,x_2,s)\coloneqq\left(x_1,x_2,\frac{2 a_1 x_1-b_1+\delta_1}{2 a_1},\frac{2 a_2 x_2-b_2+\delta_2}{2 a_2}\right)\,.
\]
All $4$ choices of the branches of $\delta_i$'s satisfy the Jacobian criterion. The pair $(\mathcal{P},X)$ satisfy the conditions from Definition~\ref{defn: restriction of radical parametrization}. Hence, we can restrict $\mathcal{P}$ to $X$ and the resulting variety $\mathcal{V}_\mathcal{P}(X)$ is irreducible of dimension $2$ and is an irreducible component of $\mathcal{F}_X$. A \verb+Bertini+ computation shows that $\mathcal{F}_X$ is irreducible, so $\mathcal{V}_\mathcal{P}(X)=\mathcal{F}_X$.

The same ideal $\mathcal{F}_X$ might be computed using the approach of Example~\ref{ex: ellipse using radical parametrization}. Indeed, the cardioid $X$ is parametrized by the map $\psi(t)=(\frac{2t^2-2}{(1+t^2)^2},\frac{-4t}{(1+t^2)^2})$. One verifies that the normal line $N_{\psi(t)}X$ is spanned by $\beta=(3t^2-1, t(t^2-3))$. Solving by radicals the system $a_i(u_i-\psi_i(t))^2+b_i(u_i-\psi_i(t))+c_i=s\beta_i$ for all $i\in\{1,2\}$ with respect to $u_i$ yields the map $\phi(t,s)=(\psi_1(t),\psi_2(t),\zeta_1(t,s),\zeta_2(t,s))$, where
\begin{align*}
\zeta_1(t,s) & = \frac{2a_1\psi_1(t)-b_1+\sqrt{b_1^2-4a_1c_1+4a_1s(3t^2-1)}}{2a_1}\\
\zeta_2(t,s) & = \frac{2a_2\psi_2(t)-b_2+\sqrt{b_2^2-4a_2c_2+4a_2st(t^2-3)}}{2a_2}\,.
\end{align*}
We introduce the parameters $\delta_1$ and $\delta_2$ via the relations $\delta_1^2=b_1^2-4a_1c_1+4a_1s(3t^2-1)$ and $\delta_2^2=b_2^2-4a_2c_2+4a_2st(t^2-3)$. Then we construct the radical tower
\[
\mathbb{T}\coloneqq[k_0\coloneqq\C(t,s) \subset k_1\coloneqq\C(t,s)(\delta_1)\subset k_2\coloneqq\C(t,s)(\delta_1,\delta_2)]\,.
\]
After specifying the evaluation of parametrization in the form $\delta_1(t_0,s_0)=q_1$ and $\delta_2(t_0,s_0)=q_2$ for some $t_0,s_0,q_1,q_2\in\C$, the corresponding radical parametrization is
\begin{equation*}
    \mathcal{P}=\left\{\left(\psi_1(t),\psi_2(t),\frac{2a_1\psi_1(t)-b_1+\delta_1}{2a_1}, \frac{2a_2\psi_2(t)-b_2+\delta_2}{2a_2}\right), \mathbb{T},\ \delta_1(t_0,s_0)=q_1,\ \delta_2(t_0,s_0)=q_2\right\}\,.
\end{equation*}
The radical variety $\mathcal{V}_\mathcal{P}$ coincides with $\mathcal{V}_\mathcal{P}$ and, in turn, with $\mathcal{F}_X$.
\end{example}

\begin{corollary}\label{prop: gen critical ideal}
Let $f_u$ and $X\subset \C^n$ satisfy the assumptions of Theorem~\ref{prop: X unirational, f-correspondence parametrizable}. 
Over general data points $u\in \C^n$, the second projection $\pi_2\colon\mathcal{F}_X \to \C^n$ has finite fibers $\pi_2^{-1}(u)$ of constant cardinality. 
\end{corollary}

\begin{proof}
We will apply the Theorem on the Dimension of Fibers to each of the irreducible components of $\mathcal{F}_X$, which by Theorem~\ref{prop: X unirational, f-correspondence parametrizable} are $n$-dimensional. If the dimension of the second projection of a component is less than $n$, then the fiber over a general $u \in \C^n$ is empty. If the dimension of the second projection of a component is $n$, then by the Theorem on the Dimension of Fibers, the fibers are generically finite sets. Since this is true for each of the components of $\mathcal{F}_X$, then also the second projection $\pi_2\colon\mathcal{F}_X \to \C^n$ has finite fibers $\pi_2^{-1}(u)$ of constant cardinality.
\end{proof}

\begin{definition}
The {\em algebraic degree} of $X\subset\C^n$ with respect to the objective function $f_u$ is the cardinality of the fiber $\pi_2^{-1}(u)$ over a general data point $u\in V$. It is denoted by $\deg_{f_u}(X)$.
\end{definition}

When $f_u = \langle u,x \rangle$, then the algebraic degree of this optimisation problem is determined by the degree of a certain polynomial that vanishes on the dual variety $X^\vee$. We refer the reader to~\cite{rostalski2012chapter} for a detailed discussion.

\begin{remark}
Looking back at Definition \ref{def: optimization correspondence}, if $X$ is contained in the variety $H=\mathbb{V}\left(\prod_{i=1}^n\left(\frac{\partial f_u}{\partial x_i}\right)_{D}\right)$, then the optimization correspondence $\mathcal{F}_X$ is empty. In this case, we say that $\deg_{f_u}(X)=0$. In general, we consider varieties $X$ that are not contained in $H$.
\end{remark}

\section{The \texorpdfstring{$p$}{p}-norm distance degree}\label{sec: p-norm distance degree}

In this section, we focus on $p$-norm optimization on algebraic varieties, where $p$ is a positive integer number. In particular, given a real variety $X^\mR\subset\R^n$, we study the optimization problem
\begin{equation}\label{eq: optimization problem p-norm}
\min_{x\in X^\mR}\|u-x\|_p=\left(\sum_{i=1}^n|u_i-x_i|^p\right)^{\frac{1}{p}}\,.
\end{equation}
When $p$ is even, this is equivalent to the optimization problem \eqref{eq: optimization problem, general form} with $f_u=\sum_{i=1}^n(u_i-x_i)^p$. Again the prototype example is for $p=2$ and corresponds to Euclidean Distance optimization. Things are subtler when $p$ is odd, as for each $i$ the quantity $|u_i-x_i|^p$ is in general different from $(u_i-x_i)^p$. Then the global minimum is to be found among the solutions of all optimization problems of the form \eqref{eq: optimization problem, general form} where $f_u=\pm(u_1-x_1)^p\pm\cdots\pm(u_n-x_n)^p$ for all choices of signs $\pm$, up to a simultaneous change of sign. All in all, there are $2^{n-1}$ optimization problems to be solved. For a concrete example of $p$-norm optimization with $p$ odd see Example \ref{ex: 3-norm optimization ellipse}. 

Taking into account the previous discussion and the warnings about the odd case, for all integer $p\ge 1$ we consider the relaxed optimization problem \eqref{eq: optimization problem, general form} with $f_u=d_u^p\coloneqq\sum_{i=1}^n(u_i-x_i)^p\in\C[u][x]$.
This allows us to apply the more general theory discussed in the previous sections.
We denote by $q_p$ the polynomial $q_p(x)=\sum_{i=1}^n x_i^p$ and we define the $p${\em -isotropic hypersurface} of $\C^n$ as $Q_p\coloneqq\mathbb{V}(q_p)\subset\C^n$. For $p$ even, its unique real point is the origin.

Let $X\subset\C^n$ be an affine variety of codimension $c$, and consider a general data point $u\in\C^n$. We use the shorthand $(u-x)^j$ to indicate the vector $((u_1-x_1)^j,\ldots,(u_n-x_n)^j)$ for any nonnegative integer $j$. In particular $\nabla d_u^p=(u-x)^{p-1}$. Hence the critical points of $d_u^p$ on $X$ are the smooth points $x\in X_\mathrm{sm}$ such that $(u-x)^{p-1}\perp T_xX$. Algebraically speaking, the critical points are zeros of the following zero-dimensional ideal in $\C[x]$:
\begin{equation} \label{eq:p-norm-critical-ideal}
\left(I(X)+\left\langle\mbox{$(c+1)\times(c+1)$ minors of}
\begin{pmatrix}
\mathrm{Jac}_{X}(x)\\
(u-x)^{p-1}
\end{pmatrix}
\right\rangle\right)\colon {I(X_{\mathrm{sing}})}^\infty\,.
\end{equation}

\begin{definition}\label{def: p-norm distance degree}
The $p${\em -norm distance degree} of an algebraic variety $X\subset\C^n$ is  the algebraic degree of $X$ with respect to the polynomial $d_u^p$. It is denoted by the shorthand $\deg_p(X)$.
\end{definition}

\begin{definition}\label{def: p-norm correspondence}
The {\em $p$-norm correspondence} of $X$ is the optimization correspondence $\mathcal{F}_X^{(p)}$ associated to $f_u=d_u^p$. It is a subvariety of $\C^n\times\C^n$ defined by the ideal \eqref{eq:p-norm-critical-ideal} in the larger polynomial ring $\C[x,u]$, which we denote by $I(\mathcal{F}_X^{(p)})$.
\end{definition}

\begin{example}[$p$-norm distance degree of a smooth plane curve]\label{ex: p-norm degree smooth plane curve}
Let $X=\mathbb{V}(g)\subset\C^2$ be a smooth plane curve of degree $d$. Mimicking the argument used in \cite[Example 4.5]{DHOST}, we have that
\begin{equation}\label{eq: ideal p-norm correspondence affine plane curve}
I(\mathcal{F}_X^{(p)}) = \left\langle g(x), \frac{\partial g(x)}{\partial x_1}(u_2-x_2)^{p-1}-\frac{\partial g(x)}{\partial x_2}(u_1-x_1)^{p-1}\right\rangle\,.
\end{equation}
For a general $u\in\C^2$, the previous ideal is zero-dimensional and defines the intersection of $X$ with a curve of degree $d+p-2$, hence $\deg_p(X)=d(d+p-2)$.
\end{example}

For a general affine complete intersection $X\subset\C^n$ we can give an upper bound on the $p$-norm distance degree of $X$ in terms of the degrees of the polynomials that cut out $X$. This is an immediate consequence of \cite[Theorem 2.2]{nie2009algebraic}. For $p=2$, the result is refined in \cite[Proposition 2.6]{DHOST}.

\begin{proposition}\label{prop: upper bound p-degree complete intersection}
Let $X=\mathbb{V}(g_1,\ldots,g_c)\subset\C^n$ be a general complete intersection of codimension $c$, where $\deg(g_i)=d_i$ for all $i\in\{1,\ldots,n\}$. Then
\[
\deg_p(X)\le d_1\cdots d_c\sum_{i_0+\cdots+i_c=n-c}(p-1)^{i_0}(d_1-1)^{i_1}\cdots(d_c-1)^{i_c}\,.
\]
\end{proposition}

Observe that, for a general affine hypersurface $X\subset\C^n$ of degree $d$, the upper bound of Proposition \ref{prop: upper bound p-degree complete intersection} simplifies to
\begin{equation}\label{eq: upper bound p-degree hypersurface}
\deg_p(X)\le d\frac{(d-1)^n-(p-1)^n}{d-p}\,.
\end{equation}
\begin{example}\label{ex: 3-norm optimization ellipse}
Consider the ellipse $X^\mR\colon x_1^2+4x_2^2-1$ in $\R^2$. Our goal is to solve the problem \eqref{eq: optimization problem p-norm} for $p=3$ and for a given data point $u\in\R^2$. In particular, we need to solve the two optimization problems
\begin{equation}\label{eq: optimization problem p-norm, example}
\min_{x\in X^\mR}d_u^{3,+}\coloneqq[(u_1-x_1)^3+(u_2-x_2)^3] \quad\mbox{and}\quad \min_{x\in X^\mR}d_u^{3,-}\coloneqq[(u_1-x_1)^3-(u_2-x_2)^3]\,.
\end{equation}
If we pick for example the point $u=(-6/10,6/10)$, the first optimization problem produces $6=\deg_p(X)$ complex critical points, which we computed with \verb+M2+. Two of them are the real points $A_1=(-0.228, -0.487)$ and $A_2=(0.998,0.032)$. The second optimization problem gives other two real critical points $A_3=(-0.508, 0.431)$ and $A_4=(0.997, -0.040)$. The first image of Figure \ref{fig: ellipse3norm} shows all real critical points $A_1,\ldots,A_4$. The dashed lines describe the boundaries of the $3$-norm circles $C_i\coloneqq\{x\colon\|x-u\|_3\le\|A_i-u\|_3\}$ for all $i\in\{1,\ldots,4\}$. The four curves $C_i$ are not algebraic, indeed the relaxed optimization problems in \eqref{eq: optimization problem p-norm, example} determine four algebraic curves
\begin{align*}
D_1&\coloneqq\mathbb{V}(d_u^{3,+}(x)-d_u^{3,+}(A_1)),\ D_2\coloneqq\mathbb{V}(d_u^{3,+}(x)-d_u^{3,+}(A_2)),\\
D_3&\coloneqq\mathbb{V}(d_u^{3,-}(x)-d_u^{3,-}(A_3)),\ D_4\coloneqq\mathbb{V}(d_u^{3,-}(x)-d_u^{3,-}(A_4))\,.
\end{align*}
The four curves $D_i$ are tangent to $X^\mR$ at the corresponding critical points $A_i$ and only partially overlap the corresponding $3$-norm balls $B_i$. The global minimum of the $3$-norm is attained at $A_3$ and is equal to $\|u-A_3\|_3\approx 0.178$.

\begin{figure}[ht]
	\centering
	\includegraphics[width=3.3in]{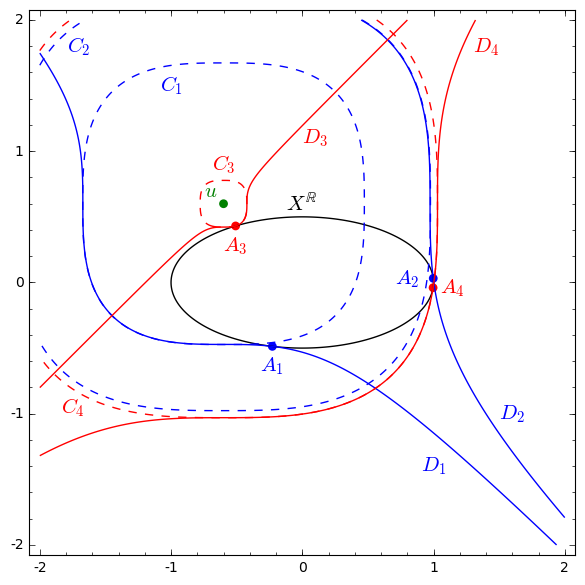}
	\includegraphics[width=3.3in]{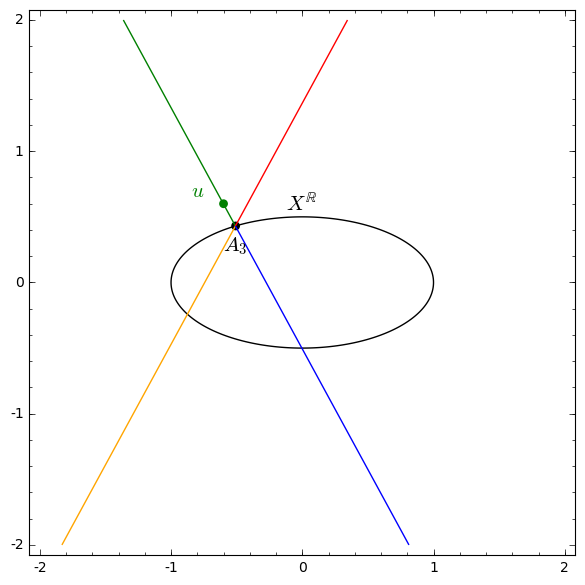}
	\caption{On the left, the critical points $A_i$ on the ellipse $X^\mR$ for the two optimization problems in \eqref{eq: optimization problem p-norm, example}, with the $3$-norm circles $C_i$ and the corresponding curves $D_i$. On the right, the four half-lines corresponding to the branches of the fiber $\pi^{-1}(A_3)$ with respect to the second optimization problem in \eqref{eq: optimization problem p-norm, example}.}\label{fig: ellipse3norm}
\end{figure}
\end{example}

\begin{example}\label{ex: fiber first projection ellipse 3-norm}
In this example we consider again the ellipse $X^\mR\colon x_1^2+4x_2^2-1$ in $\R^2$ and the optimization problems in \eqref{eq: optimization problem p-norm, example}. Here we fix a point $x\in X$ and we look for all data points $u\in\C^2$ such that $x$ is a critical point for either the function $d_u^{3,+}$ or $d_u^{3,-}$. We consider first $d_u^{3,+}$. In other words, we want to compute the fiber $\pi^{-1}(x)$ of the projection $\pi\colon\mathcal{F}_X^{(p)}\to X$. Similarly as in \eqref{eq: ideal p-norm correspondence affine plane curve}, the variety $\mathcal{F}_X^{(p)}$ is defined by the ideal
\[
\langle x_1^2+4x_2^2-1, x_1(u_2-x_2)^2-4x_2(u_1-x_1)^2\rangle\,.
\]
In particular $\pi^{-1}(x)=\mathbb{V}(x_1(u_2-x_2)^2-4x_2(u_1-x_1)^2)$ is a degenerate plane conic with discriminant $\Delta=-4x_1x_2$. For example if $x\in X^\mR$ lives in the interior of the first orthant, then we can write it as $x=(\xi_1^2,\xi_2^2)$ and hence $\pi^{-1}(x)=\mathbb{V}(\xi_1(u_2-\xi_2^2)-2\xi_2(u_1-\xi_1^2))\cup\mathbb{V}(\xi_1(u_2-\xi_2^2)+2\xi_2(u_1-\xi_1^2))$ is a union of two real lines. Or if $x\in X^\mR$ lives in the interior of the second orthant, then we can write it as $x=(-\xi_1^2,\xi_2^2)$ and hence $\pi^{-1}(x)=\mathbb{V}(\xi_1^2(u_2-\xi_2^2)^2+4\xi_2^2(u_1+\xi_1^2)^2)$ is a union of two complex conjugate lines meeting at the real point $x$. A similar argument with changed signs applies to the objective function $d_u^{3,-}$. This is intimately related to Example \ref{ex: 3-norm optimization ellipse}. Indeed, observe that the critical point $A_3$ in Example \ref{ex: 3-norm optimization ellipse} lives in the second orthant. On one hand, this implies that $A_3$ is the only real point in the fiber $\pi^{-1}(A_3)$ with respect to the first optimization problem in \eqref{eq: optimization problem p-norm, example}. On the other hand, the fiber $\pi^{-1}(A_3)$ with respect to the second optimization problem consists of two real lines, one of them containing the data point $u=(-6/10,6/10)$ as showed in the second image of Figure \ref{fig: ellipse3norm}. When moving away from $A_3$, there are 4 possible branches of data points to follow corresponding to the 4 half-lines whose union is $\pi^{-1}(A_3)$, and the choice of a specific branch is exactly determined by the radical parametrization given in Example \ref{ex: ellipse using implicit parametrization}.
\end{example}

\begin{example}
Consider the twisted cubic curve $X^\mR\subset\R^3$ depicted in green in Figure \ref{fig: fiberstwistedcubic}. It is the image of the parametrization $\psi\colon\R\to\R^3$ given by $\psi(t)=(t,t^2,t^3)$ and it is defined implicitly by the ideal $I=\langle x_2^2-x_1x_3,x_1x_2-x_3,x_1^2-x_2\rangle\subset\R[x]$. The $p$-norm correspondence $\mathcal{F}_X^{(p)}$ is defined by the ideal
\begin{equation}\label{eq: ideal p-correspondence twisted cubic}
I+\left\langle\mbox{$3\times 3$ minors of}
\begin{pmatrix}
(u_1-x_1)^{p-1} & (u_2-x_2)^{p-1} & (u_3-x_3)^{p-1}\\
-x_3 & 2x_2 & -x_1\\
x_2 & x_1 & -1\\
2x_1 & -1 & 0
\end{pmatrix}\right\rangle\,.
\end{equation}
Given a point $x\in X$, the fiber $\pi^{-1}(x)$ of the projection $\pi\colon\mathcal{F}_X^{(p)}\to X$ is a hypersurface in the complex space $\C^3$ of degree $p-1$. In Figure \ref{fig: fiberstwistedcubic} we show the fiber $\pi^{-1}(1,1,1)$ for $p\in\{2,4,6\}$. When $p=2$, the fiber coincides with the normal plane of $X$ at $(1,1,1)$. For $p$ odd, the unique real point of $\pi^{-1}(1,1,1)$ is $(1,1,1)$ itself. As $p\to\infty$, the fiber $\pi^{-1}(1,1,1)$ tends to a semialgebraic set defined by affine linear inequalities.
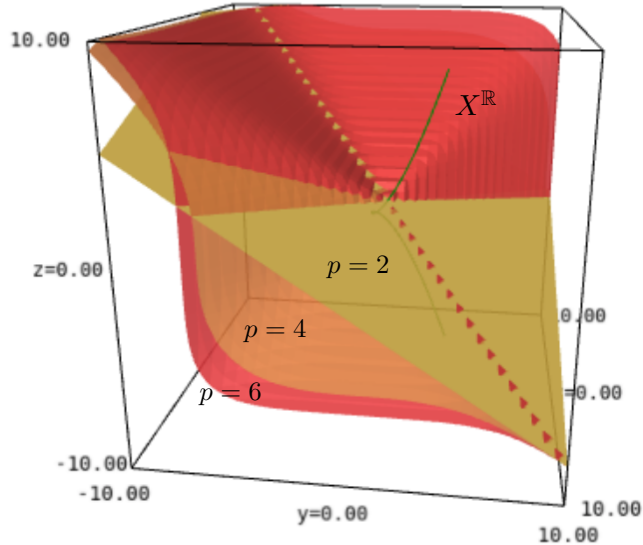
\begin{figure}[ht]
  \centering
  \begin{overpic}[width=3.3in, tics=10]{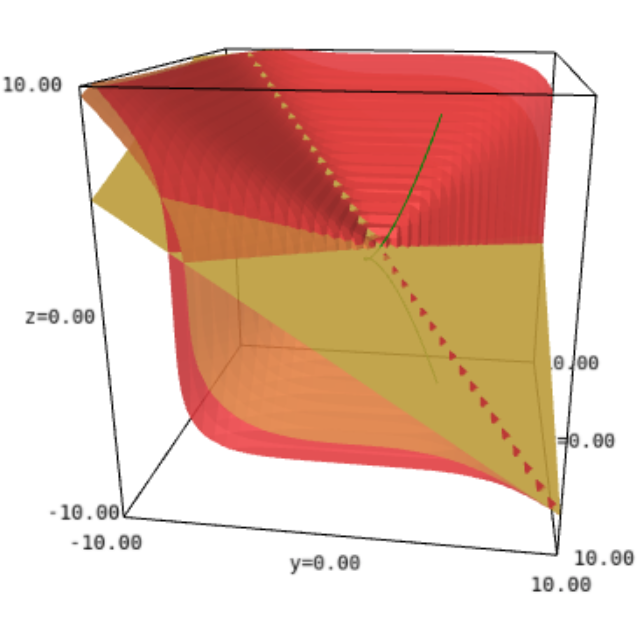}
  \put (50,50) {\small{$p=2$}}
  \put (37,40) {\small{$p=4$}}
  \put (30,30) {\small{$p=6$}}
  \put (70,75) {$X^{\mathbb{R}}$}
  \end{overpic}
  \caption{Fibers $\pi^{-1}(1,1,1)$ of $\mathcal{F}_X^{(p)}\to X$ to the twisted cubic $X$ for $p\in\{2,4,6\}$.}\label{fig: fiberstwistedcubic}
\end{figure}

\end{example}

\begin{remark}
One might naturally define the analogues of the $\varepsilon$-offset hypersurface \cite{HW} and the ED polynomial \cite{OS} in the setting of $p$-norm distance optimization over a real algebraic variety. We leave it for further research.
As an encouraging remark, we stress that these objects are important when dealing with real critical points on a real variety $X^\mR\subset\R^n$ with respect to the $p$-norm distance function. Similarly to what happens in the ED case, the branch locus of the projection $\mathcal{F}_X^{(p)}\to\C_u^n$ is typically a hypersurface in $\R^n$, which we might call {\em $p$-norm distance discriminant}. When $n=2$, the $p$-norm distance discriminant generalizes the classical {\em evolute} of real plane curves.
\end{remark}

\begin{example}
Consider the affine conic $X^\mR\colon x_1^2+4x_2^2-1=0$ in $\R^2$. A \verb|Macaulay2| computation reveals that, for $p\in\{2,3,4\}$, the ``$p$-norm evolute'' is one of the red curves displayed in Figure \ref{fig: generalized evolutes ellipse}. The classical evolute of $X$ appears on the left and has degree $6$. It divides the real plane in two regions, an inner one and an outer one where there are $4$ and $2$ real critical points, respectively. In the middle we have the $3$-norm evolute of degree $12$. We stress that this curve corresponds to the first relaxed optimization problem in \eqref{eq: optimization problem p-norm, example}. The evolute corresponding to the second optimization problem in \eqref{eq: optimization problem p-norm, example} is symmetric to the first one with respect to the $y$ axis. Finally, the $4$-norm evolute on the right has degree $18$. This suggests that the $p$-norm evolute of $X^\mR$ has degree $6(p-1)$.

Geometrically, the $p$-norm evolute of $X^\mR$ corresponds to the envelope of the family $\{(\pi^{-1}(x))^\mR\mid x\in X^\mR\}$ of real fibers of the projection $\pi\colon\mathcal{F}_X^{(p)}\to X$. Each fiber $\pi^{-1}(x)$ is defined by the second generator in the ideal \eqref{eq: ideal p-norm correspondence affine plane curve}, where in this case $g=x_1^2+4x_2^2-1$. In particular for $p=2$ we are considering the family of normal lines at the points of $X^\mR$. For $p$ even, each fiber $(\pi^{-1}(x))^\mR$ consists of one real line, whereas for $p$ odd either $(\pi^{-1}(x))^\mR=\{x\}$ or it consists of two real lines, possibly coincident. A sample of the family $\{(\pi^{-1}(x))^\mR\mid x\in X^\mR\}$ is depicted in Figure \ref{fig: generalized evolutes ellipse} for small values of $p$.
\begin{figure}[ht]
    \begin{overpic}[width=2.2in, tics=10]{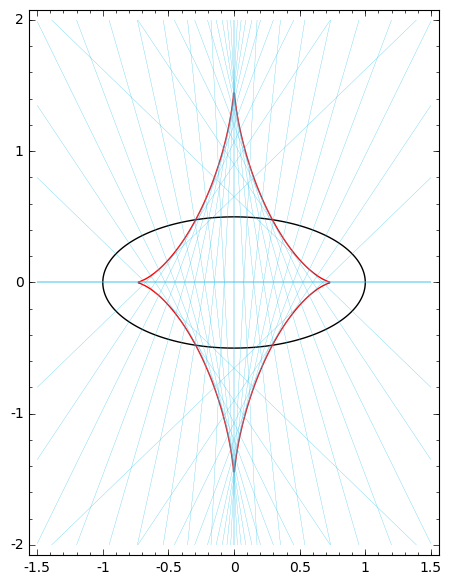}
    \put (38.5,50) {\red{$4$}}
    \put (57,80) {\red{$2$}}
    \end{overpic}
    \begin{overpic}[width=2.2in, tics=10]{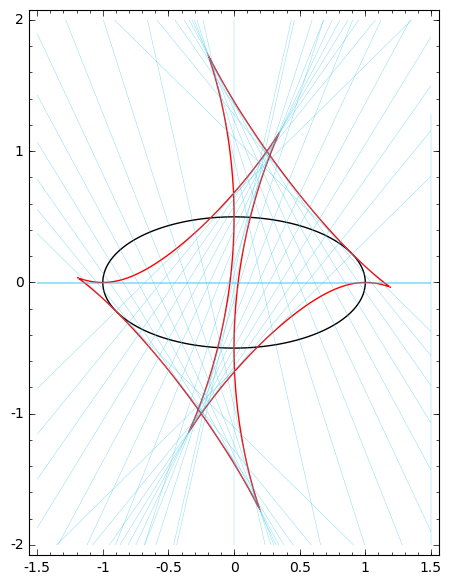}
    \put (41,72) {{\small\red{$4$}}}
    \put (37,29) {{\small\red{$4$}}}
    \put (47,55) {{\small\red{$4$}}}
    \put (30,47) {{\small\red{$4$}}}
    \put (40.5,64) {{\footnotesize\red{$6$}}}
    \put (37,36.5) {{\footnotesize\red{$6$}}}
    \put (57,80) {{\small\red{$2$}}}
    \put (47,73) {{\tiny\red{$\leftarrow\!2$}}}
    \put (27,29) {{\tiny\red{$2\!\rightarrow$}}}
    \end{overpic}
    \begin{overpic}[width=2.2in, tics=10]{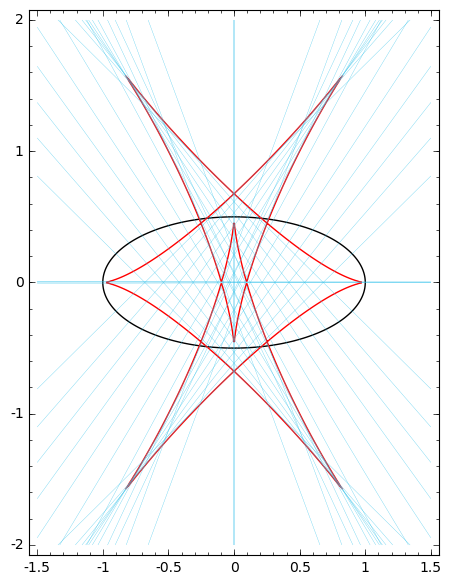}
    \put (39,63.5) {{\tiny\red{$6$}}}
    \put (39,37.5) {{\tiny\red{$6$}}}
    \put (39,50.6) {{\tiny\red{$8$}}}
    \put (44.5,68) {{\tiny\red{$4$}}}
    \put (33,68) {{\tiny\red{$4$}}}
    \put (33,35) {{\tiny\red{$4$}}}
    \put (44.5,35) {{\tiny\red{$4$}}}
    \put (46.5,50.6) {{\tiny\red{$4$}}}
    \put (30.5,50.6) {{\tiny\red{$4$}}}
    \put (61.5,75) {{\red{$2$}}}
    \end{overpic}
    \caption{The ``$p$-norm evolutes'' of the ellipse $X^\mR\colon x_1^2+4x_2^2-1=0$ for $p\in\{2,3,4\}$. They divide the real plane in several regions, in each of them the number of real critical points of the $p$-norm distance function from $X^\mR$ is constant.}\label{fig: generalized evolutes ellipse}
\end{figure}
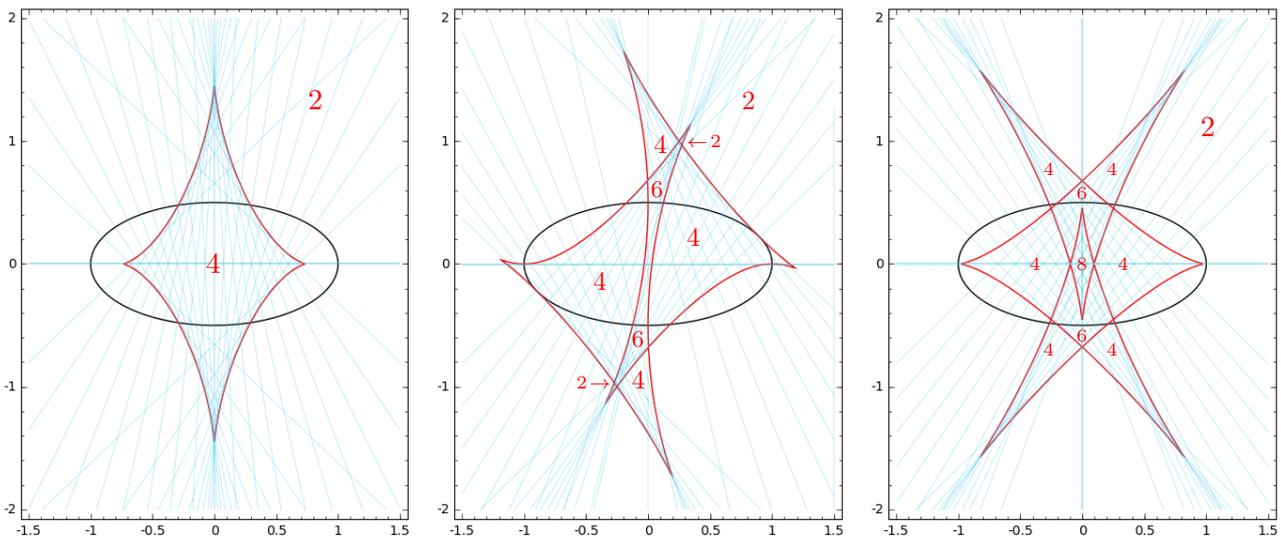 
\end{example}

For the rest of the section, we will restrict to affine cones $X\subset\C^n$ through the origin, and we will use the same notation $X$ for the corresponding projective variety in $\PP^{n-1}$. By definition, the $p$-norm distance degree of $X\subset\PP^{n-1}$ is the $p$-norm distance degree of the corresponding affine cone in $\C^n$. On one hand, the ideal $I(X)\subset\C[x]$ is homogeneous. On the other hand, the ideal $I(\mathcal{F}_X^{(p)})$ is not homogeneous. Our next goal is to replace $I(\mathcal{F}_X^{(p)})$ by a homogeneous ideal. To do this, we introduce extra variables $y=(y_1,\dots,y_n)$. 

Define
\begin{equation}
J(X) \coloneqq\left\langle\mbox{$(c+1)\times(c+1)$ minors of}
\begin{pmatrix}
y^{p-1}\\
\mathrm{Jac}_{X}(x)\\
\end{pmatrix}
\right\rangle\,,
\end{equation}
where $y^{p-1}$ denotes the vector $(y_1^{p-1},\ldots,y_n^{p-1})$.

Using the fact that $X$ and the $p$-isotropic hypersurface $Q_p$ are projective varieties, we can replace $I(\mathcal{F}_X^{(p)})$ with the following homogeneous ideal in $\C[x]$:
\begin{equation}\label{eq: final proj linearized critical ideal}
\left[\left(I(X)+J(X)+\left\langle\mbox{$3\times 3$ minors of}
\begin{pmatrix}
y\\
u\\
x\\
\end{pmatrix}
\right\rangle\right)\colon {\left[I(X_{\mathrm{sing}})\cdot I(Q_p)\cdot \langle y\rangle \right]}^\infty\, \right]\cap \C[x].
\end{equation}

The following lemma coincides with \cite[Lemma 2.8]{DHOST} when $p=2$.

\begin{lemma}\label{lemma: affine vs projective}
Consider an affine cone $X \subset\C^n$ of codimension $c$ and pick a data point $u \in \C^n\setminus X$. Let $x\in X\setminus\{0\}$ be such that the corresponding point $[x]$ in $\PP^{n-1}$ does not lie in $Q_p$. Then $[x]$ lies in the projective variety of~\eqref{eq: final proj linearized critical ideal} if and only if some scalar multiple $\lambda x$ of $x$ lies in the variety of~\eqref{eq:p-norm-critical-ideal}. Then $\lambda$ is unique. 
\end{lemma}

\begin{proof}
Both ideals in \eqref{eq:p-norm-critical-ideal} and \eqref{eq: final proj linearized critical ideal} are saturated with respect to $I(X_{\mathrm{sing}})$, so it is sufficient to prove the statement for smooth points $x \in X$ where the Jacobian $\mathrm{Jac}_{X}(x)$ has rank $c$.  
For the only-if direction, assume that $(u-\lambda x)^{p-1}$ lies in the row span of $\mathrm{Jac}_{X}(x)$. Define $y = u-\lambda x$. Then $y^{p-1}$ lies in the row span of $\mathrm{Jac}_{X}(x)$ and the dimension of the linear span of $u,x, y$ is at most two. Moreover $y \neq 0$, since $u \in V \setminus X$ and $X$ is an affine cone.  This proves the only-if direction. 

It is enough to show the if-direction for all points in the projection of the variety of~\eqref{eq: final proj linearized critical ideal} before elimination, since the Zariski closure of this projection is the variety of~\eqref{eq: final proj linearized critical ideal}.  Assume that $[x]$ lies in the projection of the variety of~\eqref{eq: final proj linearized critical ideal} before elimination, i.e., there exists $y$ such that $([x],[y])$ lies in the variety of~\eqref{eq: final proj linearized critical ideal} before elimination.  Let $x$ lie in the row span of $y$. Then $x^{p-1}=\sum_{i=1}^c\lambda_i(\mathrm{Jac}_{X}(x))_i$ for some $\lambda \in \C$. Since each $g_i$ is homogeneous, it satisfies the Euler's homogeneous function theorem, i.e.,
\[
x \cdot \nabla g_i(x)=d_i g_i(x)\,,
\]
where $d_i$ is the degree of $g_i$. 
Since $x$ is a point on the variety, $g_i(x)=0$ for all $i$, so $x \cdot \nabla g_i(x)=0$. Therefore,
\[
x \cdot x^{p-1}=x \cdot \left(\sum_{i=1}^c\lambda_i(\mathrm{Jac}_{X}(x))_i\right)=0\,,
\]
so $[x] \in Q_p$. This contradicts the hypothesis, so the matrix
$\begin{pmatrix}
y\\
x
\end{pmatrix}$
has rank $2$.  Hence $(u-\lambda x)$ lies in the row span of $y$ and $(u-\lambda x)^{p-1}$ lies in the row span of $\mathrm{Jac}_{X}(x)$ for a unique $\lambda$.
\end{proof}

Consider the closure of the image of $\mathcal{F}_X^{(p)}\cap ((\C_x^n\setminus\{0\})\times\C_u^n)$
under the map $(\C_x^n\setminus\{0\})\times\C_u^n\to\PP_x^{n-1}\times\C_u^n$ defined by $(x, u)\mapsto([x], u)$. This closure is called the {\em projective $p$-norm correspondence} of $X$, and it is denoted by $\mathcal{PF}_X^{(p)}$. It has the following properties.

\begin{theorem}
Let $X$ be an affine cone of codimension $c$ not contained in $Q_p$. The projective $p$-norm correspondence $\mathcal{PF}_X^{(p)}$ is a variety of dimension $n$ inside $\PP_x^{n-1}\times\C_u^n$. If $\mathcal{F}_X^{(p)}$ is irreducible, then  $\mathcal{PF}_X^{(p)}$ is irreducible. It is the zero set of the ideal~\eqref{eq: final proj linearized critical ideal} in $\C[x,u]$.  The general fiber of the projection $\mathcal{PF}_X^{(p)}\to\C_u^n$ is finite of constant cardinality equal to $\deg_p(X)$.
\end{theorem}

\begin{proof}
We will use Lemma~\ref{lemma: affine vs projective} to show that $\mathcal{F}_X^{(p)}$ is the variety of the ideal~\eqref{eq: final proj linearized critical ideal}. Let $x\in X\setminus(X_{\mathrm{sing}}\cup Q_p)$ and $u\in \C^n$. If $(x,u)$ lies in $\mathcal{F}_X^{(p)}$ then $([x],u)$ it is in the zero set of the ideal~\eqref{eq: final proj linearized critical ideal} and, vice versa, if $([x],u)$ is in the variety of the ideal~\eqref{eq: final proj linearized critical ideal}, then there exists $\lambda$ such that $(\lambda x,u)$ lies in $\mathcal{F}_X^{(p)}$. If $\lambda\neq0$, then $([x],u)$ lies in the projection of $\mathcal{F}_X^{(p)}$, which is contained in $\mathcal{PF}_X^{(p)}$. Otherwise $\lambda=0$. In this case we have that $u^{p-1} \perp T_xX$. For all $\varepsilon\neq 0$ define $x_\varepsilon\coloneqq\varepsilon x$ and $u_\varepsilon\coloneqq\varepsilon x + u$. Then $x_\varepsilon\in X_\mathrm{sm}$ and $(u_\varepsilon-x_\varepsilon)^{p-1}\perp T_{x_\varepsilon}X$ for all $\varepsilon\neq 0$. Here we are taking advantage of $X$ being a cone. Moreover, we have that $\lim_{\varepsilon\to0}([x_\varepsilon],u_\varepsilon)=([x],u)$, so $([x],u)$ lies in the closure of the projection of $\mathcal{F}_X^{(p)}$, which is $\mathcal{PF}_X^{(p)}$. The remaining statements are proved as in Corollary~\ref{prop: gen critical ideal}.\qedhere
\end{proof}

\section{\texorpdfstring{$p$}{p}-norm distance degrees in terms of polar classes}\label{sec: p-norm polar classes}

In this section we introduce the notions of $s$-conormal and $s$-dual varieties that generalize the classical conormal and dual varieties. The main result of this section is a formula for the $p$-norm distance degree of a projective variety $X$ transversal to the $p$-isotropic hypersurface $Q_p$, which appears as a weighted sum of the {\em polar classes} of classical algebraic geometry. There are several ways to introduce polar classes. We introduce them using the {\em conormal variety} of $X$ defined as
\begin{equation}\label{eq: conormal variety}
\mathcal{N}_X\coloneqq\overline{\{(x,y)\in \C^n\times\C^n\mid\mbox{$x\in X_{\mathrm{sm}}$ and $y\perp T_xX$}\}}\,.
\end{equation}
The conormal variety is clearly an affine cone in $\C^n\times\C^n$, where in the second factor we are identifying $\C^n$ and $(\C^n)^*$. We keep the same notation $\mathcal{N}_X$ for the corresponding variety in $\PP^{n-1}\times\PP^{n-1}$. The {\em polar classes} of $X$ are the coefficients $\delta_i(X)$ of the class in cohomology
\[
[\mathcal{N}_X]=\delta_0(X)t_x^{n-1}t_y+\delta_1(X)t_x^{n-2}t_y^2+\cdots+\delta_{n-2}(X)t_xt_y^{n-1}\in A^*(\PP^{n-1}\times\PP^{n-1})\cong\frac{\Z[t_x,t_y]}{(t_x^n,t_y^n)}\,,
\]
where $t_x=\pi_1^*([H])$, $t_y=\pi_2^*([H'])$, the maps $\pi_1,\pi_2$ are the projections onto the factors of $\PP^{n-1}\times\PP^{n-1}$ and $H, H'$ are hyperplanes in $\PP^{n-1}$. To be consistent with the standard notation, in the rest of the paper  $\delta_i(X)$ refers only to the polar classes of $X$ and in unrelated to the field extensions introduced in Section~\ref{sec: prelim radical}.

The study of critical points for the $p$-norm objective function leads to a natural generalized conormal variety.

\begin{definition}\label{def: s-conormal variety}
Let $s\ge 1$ be an integer. The {\em $s$-conormal variety} of $X$ is
\[
\mathcal{N}_X^{(s)}\coloneqq\overline{\{(x,y)\in \C^n\times\C^n\mid\mbox{$x\in X_{\mathrm{sm}}$ and $y^s\perp T_xX$}\}}\,.
\]
Moreover, the projection of $\mathcal{N}_X^{(s)}$ onto the second factor $\C^n$ is called the {\em $s$-dual variety} of $X$ and it is denoted by $Y^{(s)}$.
When $s=1$ we recover the classical conormal variety $\mathcal{N}_X$ introduced in \eqref{eq: conormal variety} and the classical dual variety $Y^{(1)}=X^\vee$ of $X$, respectively.
\end{definition}

\noindent For all $s\ge 1$, the variety $\mathcal{N}_X^{(s)}$ is the zero set of the following ideal in $\C[x,y]$:
\[
N_X^{(s)} = \left[I(X)+\left\langle\mbox{$(c+1)\times(c+1)$ minors of } \begin{pmatrix}
y^s\\ \mathrm{Jac}_X(x)
\end{pmatrix}\right\rangle\right]\colon I(X_{\mathrm{sing}})^\infty\,.
\]

\noindent It would be interesting to study which properties of the classical conormal and dual varieties extend to their $s$-generalizations.

\noindent Define the {\em joint $p$-norm correspondence} of $X$ as
\[
\mathcal{F}_{X,Y}^{(p)}\coloneqq\overline{\{(x,u-x,u)\in \C_x^n\times \C_y^n\times \C_u^n\mid\mbox{$x\in X_{\mathrm{sm}}$ and $(u-x)^{p-1}\perp T_xX$}\}}\,.
\]

The projection of $\mathcal{F}_{X,Y}^{(p)}$ into $\C_x^n\times \C_y^n$ is the $(p-1)$-conormal variety $\mathcal{N}_X^{(p-1)}$.
The joint $p$-norm correspondence $\mathcal{F}_{X,Y}^{(p)}$ is equidimensional of dimension $n$, since $\mathcal{F}_X^{(p)}$ has these properties (by Corollary~\ref{prop: gen critical ideal}), and the projection $\mathcal{F}_{X,Y}^{(p)}\to\mathcal{F}_X^{(p)}$ is birational with inverse $(x, u)\mapsto(x,u-x,u)$. If $\mathcal{F}_X^{(p)}$ is irreducible, then $\mathcal{F}_{X,Y}^{(p)}$ is irreducible.

\begin{proposition} \label{prop:p-conormal-irreducible}
If $\mathcal{F}_{X}^{(p)}$ is irreducible, then $\mathcal{N}_X^{(p-1)}$ is irreducible.
The converse is true only if $\mathcal{N}_X^{(p-1)}$ does not intersect the diagonal $\Delta\subset\PP_x^{n-1}\times\PP_y^{n-1}$.
\end{proposition}

\begin{proof}
If $\mathcal{F}_{X}^{(p)}$ is irreducible, then $\mathcal{F}_{X,Y}^{(p)}$ is irreducible as well, and $\mathcal{N}_X^{(p-1)}$ is the projection of $\mathcal{F}_{X,Y}^{(p)}$ onto $\C_x^n\times \C_y^n$. Therefore $\mathcal{N}_X^{(p-1)}$ is irreducible.

Conversely, assume that $\Delta\cap\mathcal{N}_X^{(p-1)}=\emptyset$. Denote with $\pi$ the projection of $\mathcal{F}_{X,Y}^{(p)}$ onto $\C_x^n\times \C_y^n$. In particular $\pi(\mathcal{F}_{X,Y}^{(p)})=\mathcal{N}_X^{(p-1)}$. For every $(x,y)\in \mathcal{N}_X^{(p-1)}$, the fiber $\pi^{-1}((x,y))$ consists of those data points $u\in\langle x,y\rangle$, hence it is irreducible. Moreover, the fiber has always dimension 2 because $\mathcal{N}_X^{(p-1)}\cap\Delta=\emptyset$. Therefore $\mathcal{F}_{X,Y}^{(p)}$ is irreducible and its projection $\mathcal{F}_X^{(p)}\subset\C_x^n\times \C_u^n$ is irreducible as well.
\end{proof}

We do not know a counterexample to the second sentence of Proposition~\ref{prop:p-conormal-irreducible} if the condition that $\mathcal{N}_X^{(p-1)}$ does not intersect the diagonal $\Delta\subset\PP_x^{n-1}\times\PP_y^{n-1}$ is not satisfied.

We also introduce the {\em projective joint $p$-norm correspondence} $\mathcal{PF}_{X,Y}^{(p)}$ as the closure of the image of $\mathcal{F}_{X,Y}^{(p)}\cap[(\C_x^n\setminus\{0\})\times (\C_y^n\setminus\{0\})\times \C_u^n]$ in $\PP_x^{n-1}\times \PP_y^{n-1}\times \C_u^n$.

\begin{proposition}\label{prop: projective joint p correspondence}
Let $X\subset\C^n$ be an irreducible affine cone, let $Y^{(p-1)}\subset \C^n$ be the $(p-1)$-dual variety of $X$, and assume that neither $X$ nor $Y^{(p-1)}$ is contained in $Q_p$. The variety $\mathcal{PF}_{X,Y}^{(p)}$ is equidimensional of dimension $n$ in $\PP_x^{n-1}\times\PP_y^{n-1}\times\C_u^n$. If $\mathcal{F}_{X,Y}^{(p)}$ is irreducible, then $\mathcal{PF}_{X,Y}^{(p)}$ is irreducible. It is the zero set of the tri-homogeneous ideal
\begin{equation}\label{eq: ideal proj joint p-norm corr}
\left(N_X^{(p-1)}+
\left\langle\mbox{$3\times 3$ minors of}\begin{pmatrix}x\\y\\u\end{pmatrix}\right\rangle\right)\colon
 \left\langle q_p(x) \cdot q_p(y)\right\rangle ^\infty\,.
\end{equation}
\end{proposition}

\begin{proof}
The variety $\mathcal{PF}_{X,Y}^{(p)}$ is equidimensional of dimension $n$ by construction and by the assumptions on $\mathcal{F}_{X,Y}^{(p)}$. Similarly, if $\mathcal{F}_{X,Y}^{(p)}$ is irreducible, then $\mathcal{PF}_{X,Y}^{(p)}$ is irreducible. 
We show that $\mathcal{PF}_{X,Y}^{(p)}$ is defined by the ideal \eqref{eq: ideal proj joint p-norm corr}. First, consider a point $(x,y,u)$ with $x\in X\setminus X_{\mathrm{sing}}$, $y^{p-1}\perp T_xX$ and $u=x+y$. In particular $(x,y)\in \mathcal{N}_X^{(p-1)}$ and $\dim\langle x,y,u\rangle\le 2$. Hence $([x],[y],u)$ is a zero of \eqref{eq: ideal proj joint p-norm corr} and $\mathcal{PF}_{X,Y}^{(p)}$ is contained in the variety of \eqref{eq: ideal proj joint p-norm corr}.

Before showing the other containment, observe that the points with $q_p(x)\cdot q_p(y) \neq 0$ are dense in the variety of \eqref{eq: ideal proj joint p-norm corr}, so we may assume $x,y\notin Q_p$ and $y \neq 0$.
Moreover, since $(x,y)\in\mathcal{N}_X^{(p-1)}$, we may assume that $x, y$ are nonsingular points of $X$ and $Y^{(p-1)}$, and that $y^{p-1}\perp T_xX$ and $x\perp T_{y^{p-1}}Y$, where $Y=Y^{(1)}$ is the dual variety of $X$. This implies $x\perp y^{p-1}$. If $\dim\langle x,y\rangle\le 1$, then first we conclude that $y\perp y^{p-1}$, namely $y\in Q_p$. Secondly, we have that $\dim\langle x^{p-1},y^{p-1}\rangle\le 1$, hence $x\perp x^{p-1}$, or equivalently $x\in Q_p$.

Now let $([x], [y], u)$ be in the variety of \eqref{eq: ideal proj joint p-norm corr}. Since $x,y\notin Q_p$ by our assumption, then $x$ and $y$ are linearly independent. Hence there exist unique $c,d\in\C$ such that $u=cx+dy$. If $c,d\neq0$, then we find that $(cx,dy,u)\in\mathcal{F}_{X,Y}^{(p)}\cap((\C_x^n\setminus\{0\})\times(\C_y^n\setminus\{0\})\times \C_u^n)$, thus implying $([x],[y],u)\in\mathcal{PF}_{X,Y}^{(p)}$. If $c\neq0$ but $d=0$, then define $y_\varepsilon\coloneqq\varepsilon y$ and $u_\varepsilon\coloneqq u + \varepsilon y$ for all $\varepsilon\neq 0$. Then $y_\varepsilon\neq 0$, $y_\varepsilon^{p-1}\perp T_{cx}X$ and $u_\varepsilon=cx+y_\varepsilon$, hence $(cx,y_\varepsilon,u_\varepsilon)\in\mathcal{F}_{X,Y}^{(p)}$ for all $\varepsilon\neq 0$. Taking the limit we get $\lim_{\varepsilon\to0}([cx],[y_\varepsilon],u_\varepsilon)=([x],[y],u)$. Therefore $([x],[y],u)\in\mathcal{PF}_{X,Y}^{(p)}$. A similar argument can be used in the remaining cases $d\neq0$ but $c=0$, or $c=d=0$.\qedhere
\end{proof}

We are ready to prove the main result of this section, which expresses the $p$-norm distance degree of a projective variety as a weighted sum of its polar classes. For $p=2$, the next formula coincides with \cite[Theorem 5.4]{DHOST}.

\begin{theorem}\label{thm: p-norm distance degree polar classes}
If $\mathcal{N}_X^{(p-1)}$ is irreducible and does not intersect the diagonal $\Delta\subset\PP^{n-1}\times\PP^{n-1}$, then
\[
\deg_p(X)=\sum_{j=0}^{n-2}(p-1)^{n-1-j}\delta_{n-2-j}(X)\,.
\]
\end{theorem}

\begin{proof}
Denote by $Z$ the variety of linearly dependent triples $(x,y,u)\in\C_x^n\times\C_y^n\times\C_u^n$. First, we show that $\mathcal{PF}_{X,Y}^{(p)}\subset(\mathcal{N}_X^{(p-1)}\times \C_u^n)\cap Z$. Indeed, consider a triple $(x,u-x,u)$ such that $x\neq u$, $x\in X_\mathrm{sm}$ and $(u-x)^{p-1}\perp T_xX$. If we define $y=u-x$ then clearly $(x,y)\in\mathcal{N}_X^{(p-1)}$ and the vectors $x,y,u$ are linearly dependent, therefore $(x,u-x,u)\in(\mathcal{N}_X^{(p-1)}\times \C_u^n)\cap Z$. By taking closures we get the inclusion $\mathcal{PF}_{X,Y}^{(p)}\subset(\mathcal{N}_X^{(p-1)}\times \C_u^n)\cap Z$.
On the other hand, we show that $(\mathcal{N}_X^{(p-1)}\times \C_u^n)\cap Z$ is irreducible of dimension $n$. Consider the projection $(\mathcal{N}_X^{(p-1)}\times \C_u^n)\cap Z\to\mathcal{N}_X^{(p-1)}\subset\PP_x^{n-1}\times\PP_y^{n-1}$. It is surjective and all its fibres are irreducible of dimension $2$. Here we are using the assumption that $\mathcal{N}_X^{(p-1)}\cap\Delta=\emptyset$. Since we are assuming that $\mathcal{N}_X^{(p-1)}$ is irreducible, we conclude that $(\mathcal{N}_X^{(p-1)}\times \C_u^n)\cap Z$ is irreducible as well and has dimension $2+(n-2)=n$. Since $\mathcal{PF}_{X,Y}^{(p)}$ is also $n$-dimensional, we obtain the desired equality $\mathcal{PF}_{X,Y}^{(p)}=(\mathcal{N}_X^{(p-1)}\times \C_u^n)\cap Z$. In addition, one might verify with a tangent space computation that the intersection between $(\mathcal{N}_X^{(p-1)}\times \C_u^n)$ and $Z$ is transversal, hence an open dense subset of it is a smooth scheme.

Now consider the projection $\pi\colon\mathcal{PF}_{X,Y}^{(p)}\to\C_u^n$. We know that $\pi$ has general finite fibers of constant cardinality equal to $\deg_p(X)$. Moreover, by generic smoothness \cite[Corollary III.10.7]{hartshorneAlgebraic}, the fiber $\pi^{-1}(u)$ over a general data point $u$ consists of simple points only. Taking advantage of the identity $\mathcal{PF}_{X,Y}^{(p)}=(\mathcal{N}_X^{(p-1)}\times \C_u^n)\cap Z$, the fiber $\pi^{-1}(u)$ is scheme-theoretically equal to $\mathcal{N}_X^{(p-1)}\cap Z_u$, where $Z_u$ is the fiber in $Z$ over $u$. The cardinality of this intersection is the coefficient of $t_x^{n-1}t_y^{n-1}$ in the product $[\mathcal{N}_X^{(p-1)}]\cdot[Z_u]$ in $A^*(\PP_x^{n-1}\times\PP_y^{n-1})\cong\frac{\Z[t_x,t_y]}{(t_x^n,t_y^n)}$.  
A dimension computation shows that $Z_u$ has codimension $n-2$ in $\PP_x^{n-1}\times\PP_y^{n-1}$. Its cohomology class is
\begin{equation}\label{eq: cohomology class Z_u}
[Z_u] = \sum_{j=0}^{n-2}t_x^{n-2-j}t_y^j\in\frac{\Z[t_x,t_y]}{(t_x^n,t_y^n)}\,.
\end{equation}
In order to compute $[\mathcal{N}_X^{(p-1)}]$, consider an additional projective space $\PP_z^{n-1}$ and define the variety $W\subset \PP_y^{n-1}\times\PP_z^{n-1}$ of pairs $([y],[z])$ such that $\dim\langle y^{p-1},z\rangle\le 1$. Then let $\mathcal{N}_X\subset\PP_x^{n-1}\times\PP_z^{n-1}$ be the classical conormal variety and consider the projection $(\mathcal{N}_X\times\PP_y^{n-1})\cap(\PP_x^{n-1}\times W)\to\mathcal{N}_X^{(p-1)}\subset\PP_x^{n-1}\times\PP_y^{n-1}$. This map is surjective and generically one-to-one. This means that $[\mathcal{N}_X^{(p-1)}]$ can be computed from
\[
[(\mathcal{N}_X\times\PP_y^{n-1})\cap(\PP_x^{n-1}\times W)]=[(\mathcal{N}_X\times\PP_y^{n-1})]\cdot[(\PP_x^{n-1}\times W)]\in\frac{\Z[t_x,t_y,t_z]}{(t_x^n,t_y^n,t_z^n)}\,,
\]
On one hand, the variety $W$ has codimension $n-1$ in $\PP_y^{n-1}\times\PP_z^{n-1}$ and its cohomology class is
\begin{equation}\label{eq: cohomology class W}
[W] = \sum_{i=0}^{n-1}(p-1)^{n-1-i}t_y^{n-1-i}t_z^i\,.
\end{equation}
Both identities in \eqref{eq: cohomology class Z_u} and in \eqref{eq: cohomology class W} are special instances of \cite[Corollary 16.27]{millerCombinatorial}. On the other hand,
\[
[\mathcal{N}_X\times\PP_y^{n-1}] = \sum_{k=0}^{n-2}\delta_k(X)t_x^{n-1-k}t_z^{k+1}\,.
\]
Therefore
\[
[(\mathcal{N}_X\times\PP_y^{n-1})\cap(\PP_x^{n-1}\times W)]=\sum_{k=0}^{n-2}\sum_{i=0}^{n-1}(p-1)^{n-1-i}\delta_k(X)t_x^{n-1-k}t_y^{n-1-i}t_z^{k+i+1}\,.
\]
Since $\mathcal{N}_X^{(p-1)}$ is obtained by projecting $(\mathcal{N}_X\times\PP_y^{n-1})\cap(\PP_x^{n-1}\times W)$ onto $\PP_x^{n-1}\times\PP_y^{n-1}$, this corresponds to imposing $k+i+1=\dim(\PP_z^{n-1})=n-1$ in the previous sum, thus getting
\begin{equation}\label{eq: class [Z_u]}
[\mathcal{N}_X^{(p-1)}]=\sum_{k=0}^{n-2}(p-1)^{k+1}\delta_k(X)t_x^{n-1-k}t_y^{k+1}\,.
\end{equation}
Summing up, we have
\begin{align*}
\deg_p(X) & = [Z_u]\cdot[\mathcal{N}_X^{(p-1)}]\\
& = \sum_{j=0}^{n-2}\sum_{k=0}^{n-2}(p-1)^{k+1}\delta_k(X)t_x^{2n-3-j-k}t_y^{j+k+1}\\
& = \left[\sum_{j=0}^{n-2}(p-1)^{n-1-j}\delta_{n-2-j}(X)\right]t_x^{n-1}t_y^{n-1}\,.\qedhere
\end{align*}
\end{proof}

\begin{remark}
More generally, $\deg_p(X)$ is the contribution of the projective $p$-correspondence to the intersection number computed in Theorem \ref{thm: p-norm distance degree polar classes}. If $X$ is in sufficiently general position, then this contribution in fact equals the whole intersection number.
\end{remark}

\begin{theorem}\label{thm: p-degree Chern classes}
Let $X\subset\PP^{n-1}$ be a smooth projective variety transversal to $Q_p$. Then
\[
\deg_p(X) = \sum_{k=0}^m(-1)^{k}\deg(c_{k}(X))(p^{m+1-k}-1)\,,
\]
where $c_{k}(X)$ denotes the $k$-th Chern class of $X$ for all $k\in\{0,\ldots,m\}$.
\end{theorem}

\begin{proof}
Let $m=\dim(X)$. Knowing the relations between polar classes and Chern classes (see \cite[Eq. (3)]{Holme})
\[
\delta_i(X)=\sum_{k=i}^m(-1)^{m-k}\binom{k+1}{i+1}\deg(c_{m-k}(X))
\]
we obtain that
\begin{align*}
    \deg_p(X) & = \sum_{j=0}^{n-2}(p-1)^{n-1-j}\delta_{n-2-j}(X)\\
    & = \sum_{j=0}^{n-2}(p-1)^{n-1-j}\left[\sum_{k=n-2-j}^m(-1)^{m-k}\binom{k+1}{n-1-j}\deg(c_{m-k}(X))\right]\\
    & = \sum_{j=0}^{n-2}(p-1)^{n-1-j}\left[\sum_{k=0}^m(-1)^{m-k}\binom{k+1}{n-1-j}\deg(c_{m-k}(X))\right]\\
    & = \sum_{k=0}^m(-1)^{m-k}\deg(c_{m-k}(X))\left[\sum_{j=0}^{n-2}\binom{k+1}{n-1-j}(p-1)^{n-1-j}\right]\\
    & = \sum_{k=0}^m(-1)^{k}\deg(c_{k}(X))\left[\sum_{j=0}^{n-2}\binom{m+1-k}{n-1-j}(p-1)^{n-1-j}\right]\\
%    & = \sum_{k=0}^m(-1)^{k}\deg(c_{k}(X))\left[\sum_{j=0}^{n-1}\binom{m+1-k}{n-1-j}(p-1)^{n-1-j}-1\right]\\
    & = \sum_{k=0}^m(-1)^{k}\deg(c_{k}(X))\left[\sum_{j=0}^{n-1}\binom{m+1-k}{j}(p-1)^{j}-1\right]\\
    & = \sum_{k=0}^m(-1)^{k}\deg(c_{k}(X))(p^{m+1-k}-1)\,.\qedhere
\end{align*}
\end{proof}

The previous formula coincides with \cite[Theorem 5.8]{DHOST} for $p=2$ which, in turn, is an alternative formulation of the Catanese-Trifogli formula for the ED degree \cite[p. 6026]{cataneseFocal}. The smoothness hypothesis can be dropped by substituting Chern classes $c_k(X)$ with Chern-Mather classes $c_k^M(X)$, similarly as in \cite[Proposition 2.9]{aluffiProjective}.

There are several formulas that descend from Theorem \ref{thm: p-degree Chern classes}. First, we consider smooth irreducible projective curves. The next result coincides with \cite[Corollary 5.9]{DHOST} for $p=2$.

\begin{corollary}
Let $X$ be a smooth irreducible curve of degree $d$ and genus $g$ in $\PP^{n-1}$, and suppose that $X$ is transversal to $Q_p$. Then
\[
\deg_p(X) = (p-1)[(p+1)d+2g-2]\,.
\]
\end{corollary}
\begin{proof}
The statement is a direct application of Theorem \ref{thm: p-degree Chern classes}, where $\deg(c_0(X))=d$ and $\deg(c_1(X))=2-2g$.
\end{proof}

Secondly, we consider general projective hypersurfaces. The following formula agrees with \cite[Eq. (7.1)]{DHOST} for $p=2$. Note the difference by a factor $p-1$ with respect to the upper bound of \eqref{eq: upper bound p-degree hypersurface} for general affine hypersurfaces. The proof of the identity \eqref{eq: identity Sergey} is due to Sergey Yurkevich.

\begin{proposition}\label{prop: f-degree p-norm general hypersurface}
Let $X\subset\PP^{n-1}$ be a general projective hypersurface of degree $d$. Then
\begin{equation}\label{eq: p-norm degree hypersurface}
\deg_p(X)=d(p-1)\,\frac{(d-1)^{n-1}-(p-1)^{n-1}}{d-p}\,.
\end{equation}
\end{proposition}

\begin{proof}
The Chern polynomial of $X$ is
\[
c(X)=\sum_{k=0}^{n-2}c_k(X)\cdot h^k =\frac{(1+h)^n}{1+d\,h}=\sum_{k=0}^{n-2}\left[\sum_{i=0}^k\binom{n}{i}(-d)^{k-i}\right]\cdot h^i\,,
\]
where we are computing modulo $\langle h^{n-1}\rangle$. It is derived from the short exact sequence of sheaves
\[
0\to \mathcal{T}X\to \mathcal{T}(\PP^{n-1}|X)\to \mathcal{N}(X/\PP^{n-1})\to 0\,,
\]
and applying Whitney formula, where $\mathcal{T}X$ is the tangent bundle of $X$ and $\mathcal{N}(X/\PP^{n-1})$ is the normal bundle of $X$ in $\PP^{n-1}$. Since $X$ is a degree $d$ subvariety in $\PP^{n-1}$, $\deg(h^k)=d$. As a consequence, we have $\deg(c_k(X))=d\sum_{i=0}^k\binom{n}{i}(-d)^{k-i}$. Applying Theorem \ref{thm: p-degree Chern classes} we get
\[
\deg_p(X) = d\sum_{k=0}^{n-2}(p^{n-1-k}-1)\sum_{i=0}^k\binom{n}{i}(-1)^id^{k-i}\,,
\]
hence it remains to show that
\begin{equation}\label{eq: identity Sergey}
\sum_{k=0}^{n-2}(p^{n-1-k}-1)\sum_{i=0}^k\binom{n}{i}(-1)^id^{k-i} = (p-1)\,\frac{(d-1)^{n-1}-(p-1)^{n-1}}{d-p}\,.
\end{equation}
Changing the order of summation at the left-hand side and then using the geometric series formula, we get
\begin{align*}
\sum_{k=0}^{n-2}(p^{n-1-k}-1) \sum_{i=0}^{k}\binom{n}{i}(-1)^{i} d^{k-i} &=\sum_{k=0}^{n-2} \sum_{i=0}^{k}\binom{n}{i}(-1)^{i} p^{n-1-k} d^{k-i}-\sum_{k=0}^{n-2} \sum_{i=0}^{k}\binom{n}{i}(-1)^{i} d^{k-i} \\
&=\sum_{i=0}^{n-2}(-1)^{i}\binom{n}{i}\sum_{k=i}^{n-2} d^{k-i} p^{n-1-k}-\sum_{i=0}^{n-2}(-1)^{i}\binom{n}{i}\sum_{k=i}^{n-2} d^{k-i} \\
&=\sum_{i=0}^{n-2}(-1)^{i}\binom{n}{i}\frac{p(d^{n-1-i}-p^{n-1-i})}{d-p}-\sum_{i=0}^{n-2}(-1)^{i}\binom{n}{i}\frac{d^{n-1-i}-1}{d-1}\,.
\end{align*}
The last expression is equal to
\[
\frac{p}{d-p}\left[\sum_{i=0}^{n-2}(-1)^{i}\binom{n}{i} d^{n-1-i}- \sum_{i=0}^{n-2}(-1)^{i}\binom{n}{i} p^{n-1-i}\right]-\frac{1}{d-1}\left[ \sum_{i=0}^{n-2}(-1)^{i}\binom{n}{i} d^{n-1-i}-\sum_{i=0}^{n-2}(-1)^{i}\binom{n}{i}\right]\,.
\]
The statement follows applying the Binomial Theorem to the previous four summands.\qedhere
\end{proof}

The next result coincides with \cite[Corollary 5.11]{DHOST} for $p=2$.

\begin{corollary}
Let $X\subset\PP^{n-1}$ be an $m$-dimensional smooth projective toric variety, with coordinates such that $X$ is transversal to $Q_p$. If $V_j$ denotes the sum of the normalized volumes of all $j$-dimensional faces of the simple lattice polytope $P$ associated with $X$, then
\[
\deg_p(X) = \sum_{k=0}^m(-1)^k(p^{m+1-k}-1)\cdot V_{m-k}\,.
\]
\end{corollary}

\begin{example}\label{ex: p-norm degree rational normal curve toric}
Consider a rational normal curve $C_d\subset\PP^d$ of degree $d$ in general coordinates. The associated polytope $P$ is a segment of integer length $d$.
The formula above yields
\[
\deg_p(C_d) = (p^2-1)\cdot V_1-(p-1)\cdot V_0 = d(p^2-1)-2(p-1) = (p-1)[(p+1)d-2]\,.
\]
For $p=2$, the previous formula simplifies to $\mathrm{EDdegree}(C_d)=3d-2$.
\end{example}

Let $X=\PP^{n_1-1}\times\cdots\times\PP^{n_k-1}$ be a Segre product of projective spaces embedded in $\PP(\mathrm{Sym}^{\omega_1}\C^{n_1}\otimes\cdots\otimes\mathrm{Sym}^{\omega_k}\C^{n_k})$ via the line bundle $\mathcal{O}(\omega_1,\ldots,\omega_k)$, where $\omega_1,\ldots,\omega_k$ are positive integers. The variety $X$ is called a {\em Segre-Veronese product} of projective spaces and parametrizes the partially symmetric tensors of rank at most one in $\mathrm{Sym}^{\omega_1}\C^{n_1}\otimes\cdots\otimes\mathrm{Sym}^{\omega_k}\C^{n_k}$. We can compute the $p$-norm distance degree of $X$ using Theorem \ref{thm: p-degree Chern classes}. The next formula coincides with the one computed in \cite[Remark 4.21]{ottaviani2020asymptotics} for $p=2$.

\begin{corollary}\label{cor: gen p-norm degree Segre}
Let $X\subset \PP(\mathrm{Sym}^{\omega_1}\C^{n_1}\otimes\cdots\otimes\mathrm{Sym}^{\omega_k}\C^{n_k})$ be the Segre-Veronese variety of partially-symmetric tensors of rank at most one, and assume that $X$ is transversal to $Q_p$. Let $N=\dim(X)=n_1+\cdots+n_k-k$. Then
\begin{equation}\label{eq: gen p-norm degree Segre}
\deg_p(X)=\sum_{j=0}^N(-1)^j(p^{N+1-j}-1)(N-j)!\left[\sum_{i_1+\cdots+i_k=j}\prod_{l=1}^k\frac{\binom{n_l}{i_l}\omega_l^{n_l-i_l-1}}{(n_l-i_l-1)!}\right]\,.
\end{equation}
\end{corollary}

\begin{proof}
The statement follows by special case of a Chern class computation made by the third author in \cite[Eq. (5.3.2)]{sodomaco2020}, which tells us that
\begin{equation}\label{eq: chern classes Segre-Veronese}
\deg(c_j(X)) = (N-j)!\left[\sum_{i_1+\cdots+i_k=j}\prod_{l=1}^k\frac{\binom{n_l}{i_l}\omega_l^{n_l-i_l-1}}{(n_l-i_l-1)!}\right]\quad\forall\,j\in\{0,\ldots,N\}\,.\qedhere
\end{equation}
\end{proof}

\begin{example}
For example, if $k=2$, $n_1=2$, $\omega_1=\omega_2=1$ and $n_2=n$, then $N=n$ and the identity \eqref{eq: gen p-norm degree Segre} simplifies to
\[
\deg_p(\PP^1\times\PP^{n-1})=\sum_{j=0}^n(-1)^j(p^{n+1-j}-1)(n-j)!\left[\frac{\binom{n}{j}}{(n-1-j)!}+\frac{2\binom{n}{j-1}}{(n-j)!}\right]\,.
\]
With a similar computation as in \cite[Remark 4.21]{ottaviani2020asymptotics}, we get
\[
\deg_p(\PP^1\times\PP^{n-1}) = (p-1)^{n-1}[np^2-2p+2]\,.
\]
\end{example}

Another special case occurs when $k=1$. The next result coincides with \cite[Proposition 7.10]{DHOST} for $p=2$.

\begin{corollary}
Let $X\subset\PP(\mathrm{Sym}^\omega\C^n)$ be an $\omega$-th Veronese embedding of $\PP^{n-1}$. Assume that $X$ is transversal to $Q_p$. Then
\[
\deg_p(X)=\frac{(\omega p-1)^n-(\omega-1)^n}{\omega}\,.
\]
\end{corollary}

\begin{proof}
By Corollary \ref{cor: gen p-norm degree Segre} we have
\[
\deg_p(X) = \sum_{j=0}^{n-1}(-1)^j(p^{n-j}-1)\binom{n}{j}\omega^{n-1-j}\,.
\]
The statement follows by the Binomial Theorem.
\end{proof}

Finally, a natural consequence of Theorem \ref{thm: p-degree Chern classes} is a formula for the $p$-norm distance degree in terms of the Euler characteristic of a special difference of varieties. The following formula is inspired by \cite[Eq. (1.1)]{AHEDdegree}.
Note that for us $c_i(X)$ is the component of the Chern class of $X$ of dimension $\dim(X)-i$ (as with standard Chern classes), while in \cite{aluffiProjective} it is the component of dimension $i$.
We think that a similar formula holds without the assumption of transversality between $X$ and the $p$-isotropic hypersurface $Q_p$.

\begin{proposition}\label{prop: p-degree smooth proj variety Euler}
Let $X\subset\PP^{n-1}$ be smooth projective variety of dimension $m$ in general position with respect to $Q_p$. Then
\begin{equation}\label{eq: p-degree smooth projective Euler}
\deg_p(X)=(-1)^m(p-1)\chi(X\setminus(Q_p\cup H))\,,
\end{equation}
where $H$ is a general hyperplane of $\PP^{n-1}$.
\end{proposition}

\begin{proof}
For completeness, we briefly recall and adapt the steps used to prove \cite[Eq. (3.2)]{AHEDdegree} when $X\subset\PP^{n-1}$ is a smooth projective variety of dimension $m$ in general position with respect to $Q_p$. First we observe that
\[
\frac{1}{(1+h)(1+ph)} = \sum_{i=0}^\infty(-1)^ih^i\sum_{j=0}^\infty(-p)^jh^j = \sum_{s=0}^\infty(-1)^s\left[\sum_{i=0}^s p^i\right]h^s = \sum_{s=0}^\infty(-1)^s\frac{p^{s+1}-1}{p-1}h^s\,,
\]
hence the coefficient of $h^N$ in the expansion of $\frac{h^{N-k}}{(1+h)(1+ph)}$ is $(-1)^{k}\frac{p^{k+1}-1}{p-1}$. Then, by Theorem \ref{thm: p-degree Chern classes},
\begin{align*}
    \deg_p(X) &= \sum_{k=0}^m(-1)^{k}\deg(c_{k}(X))(p^{m+1-k}-1)\\
    & = (-1)^m\sum_{k=0}^m\deg(c_{k}(X))(-1)^{k}(p^{k+1}-1)\\
    & = (-1)^m(p-1)\sum_{k=0}^m\deg(c_{k}(X))\int\frac{h^{m-k}}{(1+h)(1+ph)}\cdot h^{n-1-m}\cap[\PP^{n-1}]\\
    & = (-1)^m(p-1)\sum_{k=0}^m\int\frac{1}{(1+h)(1+ph)}\cdot c(\mathcal{T}X)\cap[X]\\
    & = (-1)^m(p-1)\sum_{k=0}^m\int\left(1-\frac{h}{1+h}-\frac{ph}{1+ph}+\frac{h\cdot ph}{(1+h)(1+ph)}\right)\cdot c(\mathcal{T}X)\cap[X]\,,
\end{align*}
where $c(\mathcal{T}X)\cap[X]$ is the total Chern class of $X$.
Let $H\subset\PP^{n-1}$ be a general hyperplane. Since $X$ is transversal to $Q_p$, the last expression is equal to
\[
\begin{gathered}
(-1)^m(p-1)\left[\int c(\mathcal{T}X)\cap[X]-\int c(\mathcal{T}(X\cap H))\cap[X\cap H]\right.\\
\left.-\int c(\mathcal{T}(X\cap Q_p))\cap[X\cap Q_p]+\int c(\mathcal{T}(X\cap H\cap Q_p))\cap[X\cap H\cap Q_p]\right]\,,
\end{gathered}
\]
Since the degree of the top Chern class of a compact complex nonsingular variety $Y\subset\PP^{n-1}$ is its Euler characteristic $\chi(Y)$, we conclude that
\[
\deg_p(X) = (-1)^m(p-1)[\chi(X)-\chi(X\cap H)-\chi(X\cap Q_p)+\chi(X\cap Q_p\cap H)]\,.
\]
The statement follows by the inclusion-exclusion properties of the Euler characteristic.
\end{proof}

\begin{example}\label{ex: p-degree smooth proj plane curve}
Let $X$ be a smooth curve of degree $d$ in $\PP^2$. Then its genus is $g=\binom{d-1}{2}$ and by the Riemann-Hurwitz formula its Euler characteristic is $\chi(X)=2(1-g)$, therefore $\chi(X)=d(3-d)$.
Let $H$ be a general line in $\PP^2$. We get that
\[
\chi(X\setminus(Q_p\cup H))=\chi(X)-\chi(X\cap Q_p)-\chi(X\cap H)+\chi(X\cap Q_p\cap H)\,.
\]
The intersections $X\cap Q_p$ and $X\cap H$ are unions of $pd$ and $d$ points respectively, while $X\cap Q_p\cap H=\emptyset$. If $X$ is transversal to $Q_p$, we can apply Proposition \ref{prop: p-degree smooth proj variety Euler} and conclude that
\[
\deg_p(X)=(1-p)\chi(X\setminus(Q_p\cup H))=(1-p)(d(3-d)-d(p+1))=d(p-1)(d+p-2)\,.
\]
Note the difference with the affine case of Example \ref{ex: p-norm degree smooth plane curve} by a factor $p-1$, which is $1$ for the Euclidean distance.
\end{example}

\begin{example}
In Example \ref{ex: p-norm degree rational normal curve toric} we computed the $p$-norm distance degree of a rational normal curve in $\PP^d$. We may repeat the same computation using Proposition \ref{prop: p-degree smooth proj variety Euler}. Indeed, in general $\chi(\PP^d)=d+1$, so $\chi(X)=\chi(\PP^1)=2$. Let $H$ be a general hyperplane in $\PP^d$. 
Similarly to the previous example, we have that $X\cap Q_p$ and $X\cap H$ are the union of $pd$ and $d$ points respectively, and $X\cap Q_p\cap H=\emptyset$, so we conclude that
\[
\deg_p(X)=(1-p)\chi(X\setminus(Q_p\cup H))=(p-1)[(p+1)d-2]\,.
\]
For $d=2$, the previous formula simplifies to $\deg_p(X)=2p(p-1)$ and agrees with Example \ref{ex: p-degree smooth proj plane curve}.
\end{example}

Proposition \ref{prop: p-degree smooth proj variety Euler} has an affine counterpart, which can be proved similarly as in \cite[Theorem 3.8]{MRWmultiview} for arbitrary $p$.

\begin{proposition}\label{prop: p-norm distance degree smooth affine Euler}
Let $X\subset\C^n$ be a smooth closed subvariety of dimension $m$. Let $z_1,\ldots,z_n$ be coordinates of $\C^n$. For a general $\beta=(\beta_0,\ldots,\beta_n)\in\C^{n+1}$, let $U_\beta$ denote the complement of the hypersurface $\sum_{i=1}^n(z_i-\beta_i)^p+\beta_0=0$ in $\C^n$. Then
\[
\deg_p(X)=(-1)^m\chi(X\cap U_\beta)\,.
\]
\end{proposition}

\begin{proof}
Consider the map
\[
i\colon\C^n\to\C^{pn}\,,\ (z_1,\ldots,z_n)\mapsto (z_1,\ldots,z_n,z_1^2,\ldots,z_n^2,\ldots,z_1^p,\ldots,z_n^p)\,.
\]
The map $i$ is a closed embedding of $\C^n$ into $\C^{pn}$.
Consider coordinates $(w_{11},\ldots,w_{1n},\ldots,w_{p1},\ldots,w_{pn})$ in $\C^{pn}$.
Then the function $\sum_{k=1}^n(z_k-\beta_k)^p+\beta_0$ on $C^n$ is equal to the pullback of the linear function $\beta_0+\sum_{i=0}^p\sum_{j=1}^n(-1)^{p-i}\binom{p}{i}w_{ij}\beta_j^{p-i}$.
Then conclusion follows applying \cite[Theorem 3.1]{MRWmultiview}.
\end{proof}

Proposition \ref{prop: p-norm distance degree smooth affine Euler} does not apply to smooth projective varieties because their affine cones have a singularity at the origin.

\section*{Acknowledgements}

All authors are partially supported by the Academy of Finland Grant No. 323416. We acknowledge the computational resources provided by the Aalto Science-IT project. We thank Emil Horobe\c{t}, Giorgio Ottaviani, Rafael Sendra, David Sevilla, Bernd Sturmfels, Carlos Villarino, Sergey Yurkevich and the community of StackExchange,  especially Georges Elencwajg and KReiser, for helpful suggestions.

\bibliographystyle{alpha}
\bibliography{bibliography}

\end{document}